\documentclass[11pt]{amsart}
\usepackage{amscd}
\usepackage{amsfonts}
\usepackage{amsmath,amsthm,hyperref}
\usepackage[all]{xy}
\usepackage{amsmath,amsthm,hyperref}
\usepackage{amsmath,amssymb,amsthm,latexsym}
\usepackage{xcolor}
\usepackage{amsmath}
\usepackage{bm}
\usepackage{dsfont}
\usepackage{tikz-cd}

\usepackage{enumerate}
\numberwithin{equation}{section}
\newtheorem{theorem}{Theorem}[section]
\newtheorem{proposition}[theorem]{Proposition}
\newtheorem{definition}[theorem]{Definition}
\newtheorem{corollary}[theorem]{Corollary}
\newtheorem{lemma}[theorem]{Lemma}

\theoremstyle{remark}
\newtheorem{remark}[theorem]{Remark}

\newcommand{\C}{\mathbb{C}}
\newcommand{\D}{\mathbb{D}}
\newcommand{\E}{\mathbb{E}}

\newcommand{\dd}{\mathrm{d}}
\newcommand{\FC}{\mathcal{C}}

\newcommand{\Str}{\mathbb{S}}

\begin{document}
\title[Algebraic and totally symmetric harmonic maps  ]{\bf{Algebraic Harmonic Maps, Totally Symmetric Harmonic Maps and a Conjecture}}
\author{Josef F. Dorfmeister, Peng Wang}
\maketitle

\begin{abstract}
In this paper, we discuss the associated family  of harmonic maps $\mathcal{F}: M \rightarrow G/K$   from a Riemann surface $M$ into  inner symmetric spaces of compact or non-compact type which are either algebraic or totally symmetric.
These notions are the two components of the definition of a harmonic map of finite uniton number, as stated by \cite{BuGu}.
We finish this paper by comparing these to notions and by stating a conjecture.
\end{abstract}




\section{\bf{Introduction}}

 Harmonic maps from Riemann surfaces into symmetric spaces arise naturally in geometry and mathematical physics and hence became important objects in several mathematical fields, including  the study of  minimal surfaces, CMC surfaces, Willmore surfaces and related integrable systems. For harmonic maps into compact symmetric spaces or compact Lie groups, one of the most foundational and important papers is the  {description} of all harmonic two spheres into $U(n)$ by Uhlenbeck \cite{Uh}. It was shown that harmonic two-spheres satisfy a very restrictive condition. Uhlenbeck coined the expression  ``finite uniton number'' for this property. Since the uniton number is an integer, one also obtains this way a subdivision of harmonic maps into  $U(n)$.

 Uhlenbeck's work was generalized in a very elegant way to harmonic two spheres in all compact semi-simple Lie groups  and into all compact inner symmetric spaces by Burstall and Guest in \cite{BuGu}. Using  Morse theory for loop groups in the spirit of Segal's work, they showed that  harmonic maps of finite uniton number in compact Lie groups can be related to meromorphic maps into  (finite dimensional)  nilpotent  Lie algebras. They also provided a concrete method to find all such nilpotent Lie algebras. Finally, via the Cartan embedding, harmonic maps into compact inner symmetric spaces are considered as harmonic maps into Lie groups satisfying some algebraic ``twisting'' conditions. Therefore the theory of Burstall and Guest provides a description of  harmonic maps of finite uniton number into compact Lie groups and compact inner symmetric spaces, in a way which is not only theoretically satisfying, but can also be implemented well for concrete computations \cite{BuGu,Gu2002}.

 Let us take a somewhat closer look at the paper \cite{BuGu}. The title and the introduction of loc.cit. only refer to harmonic maps $\mathcal F: S^2 \rightarrow \Str$, where $\Str$ is a compact  inner symmetric space. A large part of the body of the paper, however, deals with harmonic maps  $\mathcal F: M \rightarrow \Str$, where $\Str$ is as above and $M$ is an arbitrary, connected, compact or non-compact, Riemann surface satisfying the following conditions:
 \begin{itemize}
 \item There exists an extended solution $\Phi(z, \bar z, \lambda): M \rightarrow \Str$, $\lambda \in S^1$,
 such that
\[  \hbox{$\Phi$ has a finite uniton number $k \geq 0$.}\]
 \end{itemize}
 Such harmonic maps are said to be of finite uniton number $k$. See page 546 of \cite{BuGu} for more details.

  In the present paper we use the DPW method to investigate harmonic maps into non-compact  inner symmetric spaces. Therefore we consider extended frames, not extended solutions.  These frames are defined on the universal cover $\tilde{M}$ of $M$.  One can  include, with some caveat, the case $M = S^2$,  see Section 3 of \cite{DoWa12} and
  $(2)$ of Remark \ref{sphere}. The two conditions above defining finite uniton  number harmonic maps  translate into the following two properties of  extended frames (see
 Proposition \ref{typeequivnumber} and Proposition \ref{prop-fut}):
  \begin{theorem} \label{Theorem1.1}
  {Let  $\mathcal F: M\rightarrow \Str$ be a harmonic map from a connected Riemann surface $M$ into  a compact, inner symmetric space.
 Then $\mathcal F$ has a finite uniton number $($$k\geq 0$ for some integer $k$$)$ if and only if}
 \begin{enumerate}[(U1)]
 \item  { There exists an extended frame $F : \tilde{M} \rightarrow \Lambda G_\sigma$
 for $\mathcal F$ which has  trivial monodromy relative to the action of $\pi_1(M)$ on $\tilde{M}$.}
\item  {There exists some frame  $F : \tilde{M} \rightarrow \Lambda G_\sigma$ for $\mathcal F$,  whose Fourier expansion relative to $\lambda$ is a Laurent polynomial.}
 \end{enumerate}
Here $ \Lambda G_\sigma$  denotes the twisted loop group associated to $\Str$.
  \end{theorem}

   Note that property $(U1)$ for some extended frame is equivalent to property $(U1)$ for all extended frames.  Harmonic maps $\mathcal F$, as well as the corresponding extended frames, are said  to be of {\bf finite uniton type} if the properties $(U1)$ and $(U2)$ are satisfied.
   We will apply the notion of "finite uniton type" analogously, if $\Str$ is a non-compact symmetric space (see Definition \ref{def-uni}). A discussion of the properties $(U1)$ and $(U2)$  will be given in Section 3.   Both properties together (i.e. the case of finite uniton type harmonic maps) will be investigated in Section 4. Moreover, in
    Proposition \ref{typeequivnumber} we will show that for a compact Riemann surface $M$ harmonic maps of finite uniton type are  in a bijective relation with  finite uniton number harmonic maps  in the sense of Uhlenbeck \cite{Uh} (see also \cite{BuGu}).

{ While in the literature  primarily harmonic maps $\mathcal F: M\rightarrow G/K$ were considered, where $G/K$ is a compact symmetric space,  in the theory of Willmore surfaces in $S^n$, and many other surface classes, one has to deal with ``Gauss type maps" which are harmonic maps into non-compact symmetric spaces. Therefore we will also consider the case that $G/K$ is a non-compact symmetric space. Moreover, it is in fact the  goal of the paper \cite{DoWa-fu2} to generalize results of \cite{BuGu} to harmonic maps of finite uniton type into a non-compact inner symmetric space and this paper is the first part of this project.}

This paper is organized as follows. In Section 2 we review the basic results of the loop group theory for harmonic maps.
In Section 3, we provide a detailed description of harmonic maps of finite uniton type.
Several equivalent definitions  are given for such maps.
Moreover, we also discuss briefly the monodromy and dressing actions for harmonic maps of finite uniton type.  In Section 4, we first collect the relations between the extended frames and extended solutions associated to a harmonic map into inner symmetric spaces. Then we show that for harmonic maps
 into inner symmetric spaces, the notion ``of finite uniton type" is equivalent to the notion of ``of finite uniton number".  We then end the paper with a conjecture concerning a result of Segal by Section 5.

\section{\bf{Review of basic loop group theory}}

 For any  inner involution $\sigma$ of a semi-simple  {real} Lie group $G$ the center of $G$ is contained in the connected component of the fixed point set of $\sigma$.
To begin with, we first recall some  notation. Let $G$ be a connected real semi-simple Lie group, compact or non-compact, represented as a matrix Lie group. Let $G/K$ be an inner symmetric space, compact or non-compact, with the involution
$\sigma: G\rightarrow G$ such
that $G^{\sigma}\supset K\supset(G^{\sigma})^0$, where ``0" denotes ``identity component".
{\em  For the purposes of this paper the actual choice of $K$ will be of little importance. The reader may thus simply assume that $K = \hat{K} = G^\sigma$ holds.}
 In particular, we can assume without loss of generality that $G$ has trivial center.
 We will keep this assumption throughout this paper, except where we state the opposite.
Note that  {(the tangent bundle of) } $G/K$ carries a left-invariant non-degenerate symmetric bilinear form.
Let $\mathfrak{g}$ and $\mathfrak{k}$ denote the
Lie algebras of $G$ and $K$ respectively. The involution $\sigma$ induces
a decomposition of $\mathfrak{g}$ into eigenspaces, the (generalized) Cartan decomposition
\[\mathfrak{g}=\mathfrak{k}\oplus\mathfrak{p},\hspace{5mm} \hbox{ with }\  [\mathfrak{k},\mathfrak{k}]\subset\mathfrak{k},
~~~ [\mathfrak{k},\mathfrak{p}]\subset\mathfrak{p}, ~~~
[\mathfrak{p},\mathfrak{p}]\subset\mathfrak{k}.\]
Let $\pi:G\rightarrow G/K$ denote the projection of $G$ onto $G/K$.

Now let  $\mathfrak{g^{\mathbb{C}}}$ be the complexification of $\mathfrak{g}$ and $G^{\mathbb{C}}$  the connected complex (matrix) Lie group with Lie algebra $\mathfrak{g^{\mathbb{C}}}$.
Let $\tau$ denote the complex anti-holomorphic involution
$g \rightarrow \bar{g}$, of $G^{\mathbb{C}}$. Then $G=Fix^{\tau}(G^{\C})^0$.
The inner involution $\sigma: G \rightarrow G $ commutes with the complex conjugation $\tau$ and extends to
the complexified Lie group $G^\C$, $\sigma: G^{\mathbb{C}}\rightarrow G^{\mathbb{C}}$. Let $K^{\mathbb{C}}\subset \hbox{Fix}^{\sigma}(G^{\mathbb{C}})$ denote the smallest complex subgroup of $G^{\C}$ containing $K$. Then the Lie algebra of $K^{\mathbb{C}}$ is
$\mathfrak{k^{\mathbb{C}}}$.

Occasionally we will also use another complex anti-linear involution, $\theta$, which commutes with $\sigma$ and $\tau$ and has as fixed point set
$ Fix^{\theta}(G^{\C})$,   a  maximal compact subgroup of
$G^{\C}$. For more details on the basic setting we refer to \cite{DoWa13}.

\begin{remark}
In this paper we only consider inner symmetric spaces. However, several of our results
also hold for arbitrary symmetric spaces. To keep the presentation of the paper
as simple as possible we will not consider  { the case of outer symmetric spaces in any detail in this paper. }
\end{remark}


\subsection{\bf{Harmonic Maps into Symmetric Spaces}}

Let $G/K$ be  {an inner symmetric} space as above and let $\mathcal{F}:M\rightarrow G/K$ be a harmonic map  {where $M$ is a connected Riemann surface. Without specification, $M$ may be compact or non-compact.}
  {\bf In this paper we will always assume  that $\mathcal{F}$ is ``full"}   {in the following sense. We will mention the assumption "full" only where it seems to be particularly important.}

 {
\begin{definition} \label{deffull}
A harmonic map $\mathcal{F}:M\rightarrow G/K$ is called ``full"  if only $g = e$ fixes every element
of $\mathcal{F}(M)$.
That is,  if there exists $g\in G$ such that $g \mathcal{F}(p)=\mathcal{F}(p)$ for all
$p\in M$, then $g=e$.
\end{definition}}

 {Let $\tilde{\pi} : \tilde{M} \rightarrow M$ be the universal cover of $M$ and $z_0 \in \tilde{M}$ satisfying $\tilde{\pi}(z_0) = p_0$.
Then $\mathcal{F}$ has a natural lift $\tilde{\mathcal{F}}: \tilde{M} \rightarrow G/K$ satisfying
$\tilde{\mathcal{F}} = \mathcal{F} \circ \tilde{\pi}$ and obviously  $\tilde{\mathcal{F}}(z_0,\bar z_0) = eK$. Moreover, there exists a frame $F: \tilde{M} \rightarrow G$ such that $\tilde{\mathcal{F}}=\pi \circ F$ and $F(z_0,\bar z_0) = e.$}

Let $\alpha$ denote the Maurer-Cartan form of $F$. Then $\alpha$ satisfies the Maurer-Cartan equation
and we have
\begin{equation*}F^{-1}\mathrm{d} F= \alpha, \hspace{5mm} \mathrm{d} \alpha+\frac{1}{2}[\alpha\wedge\alpha]=0.
\end{equation*}
Decomposing $\alpha$ with respect to $\mathfrak{g}=\mathfrak{k}\oplus\mathfrak{p}$ we obtain
\[\alpha=\alpha_{ \mathfrak{k}  } +\alpha_{ \mathfrak{p} }, \
\alpha_{\mathfrak{k  }}\in \Gamma(\mathfrak{k}\otimes T^*M),\
\alpha_{ \mathfrak{p }}\in \Gamma(\mathfrak{p}\otimes T^*M).\] Moreover, considering the complexification $TM^{\mathbb{C}}=T'M\oplus T''M$, we decompose $\alpha_{\mathfrak{p}}$ further into the $(1,0)-$part $\alpha_{\mathfrak{p}}'$ and the $(0,1)-$part $\alpha_{\mathfrak{p}}''$. Set  \begin{equation} \label{alphalambda}
\alpha_{\lambda}=\lambda^{-1}\alpha_{\mathfrak{p}}'+\alpha_{\mathfrak{k}}+\lambda\alpha_{\mathfrak{p}}'', \hspace{5mm}  \lambda\in S^1.
\end{equation}

\begin{lemma}  $($Pohlmeyer Lemma, see \cite{DPW}$)$  The map  $\mathcal{F}:M\rightarrow G/K$ is harmonic if and only if
\begin{equation}\label{integr}\mathrm{d}
\alpha_{\lambda}+\frac{1}{2}[\alpha_{\lambda}\wedge\alpha_{\lambda}]=0,\ \ \hbox{for all}\ \lambda \in S^1.
\end{equation}
\end{lemma}

\begin{definition}\label{def-1} Let $\mathcal{F}:M\rightarrow G/K$ be harmonic
and  {let $F: \tilde{M} \rightarrow G$ be a frame satisfying $\mathcal{F}=\pi \circ F$
and $F(z_0,\bar z_0) = e$, stated as above.}
Define $\alpha_{\lambda}$ as   {in \eqref{alphalambda} and consider
on $\tilde{M}$  a solution  $F(z,\bar z, \lambda), \lambda \in \C^*$,  to the equation }
\begin{equation}\label{eq-F-int}
 {\mathrm{d} F(z,\bar z, \lambda)= F(z,\bar z, \lambda)\alpha_{\lambda}, \hspace{2mm} F(z_0,\bar z_0,\lambda) = e.}
\end{equation}

 {As a consequence,  the choice of initial condition stated determines $F(*,*,\lambda)$ uniquely by $F$. Moreover,
we then also obtain  $F(z,\bar z, \lambda)|_{\lambda = 1} = F(z,\bar z)$
for all $z \in \tilde M$.
Any such solution will be called an  {\em extended frame} for the harmonic map $\mathcal{F}$.}

 {Finally we would like to point out that  $\mathcal F_{\lambda}:=F(z,\bar z, \lambda)\mod K$  gives a family of harmonic maps with $\mathcal F_{\lambda}|_{\lambda=1}=\mathcal F$. It will be called the {\bf associated family of $\mathcal F$}.}
 \end{definition}

\begin{remark}
 In this paper, we will usually assume for a harmonic map the conventions introduced above.
In particular, we will assume (without further saying) the existence of a base point $z_0$ and an extended frame such that  {$F(z_0,\bar z_0, \lambda) = e$} holds.
\end{remark}

\subsection{\bf{Loop Groups and Decomposition Theorems}}

 For the construction of harmonic maps we will always employ the loop group method.  In this context we consider the twisted loop groups of $G$ and $G^{\mathbb{C}}$
and some of their frequently occurring subgroups:
\begin{equation*}
\begin{array}{llll}
\Lambda G^{\mathbb{C}}_{\sigma} ~&=\{\gamma:S^1\rightarrow G^{\mathbb{C}}~|~ ,\
\sigma \gamma(\lambda)=\gamma(-\lambda),\lambda\in S^1  \},\\[1mm]
\Lambda G_{\sigma} ~&=\{\gamma\in \Lambda G^{\mathbb{C}}_{\sigma}
|~ \gamma(\lambda)\in G, \hbox{for all}\ \lambda\in S^1 \},\\[1mm]
\Omega G_{\sigma} ~&=\{\gamma\in \Lambda G_{\sigma}|~ \gamma(1)=e \},\\[1mm]
\Lambda^{-} G^{\mathbb{C}}_{\sigma}  ~&=
\{\gamma\in \Lambda G^{\mathbb{C}}_{\sigma}~
|~ \gamma \hbox{ extends holomorphically to }  {D_\infty \}},\\[1mm]
\Lambda_{*}^{-} G^{\mathbb{C}}_{\sigma} ~&=\{\gamma\in \Lambda G^{\mathbb{C}}_{\sigma}~
|~ \gamma \hbox{ extends holomorphically to }D_{\infty},\  \gamma(\infty)=e \},\\[1mm]
\Lambda^{+} G^{\mathbb{C}}_{\sigma} ~&=\{\gamma\in \Lambda G^{\mathbb{C}}_{\sigma}~
|~ \gamma \hbox{ extends holomorphically to }D_{0} \},\\[1mm]
\Lambda_{S}^{+} G^{\mathbb{C}}_{\sigma} ~&=\{\gamma\in
\Lambda^+ G^{\mathbb{C}}_{\sigma}~|~   \gamma(0)\in S \},\\[1mm]
\end{array}\end{equation*}
where $D_0=\{z\in \mathbb{C}| \ |z|<1\}$,  { $D_\infty=\{z\in \mathbb{C}| \ |z|>1\} \cup\{\infty\}$ }and $S$ is some subgroup of $K^\C$.
\vspace{2mm}

If the group $S$ is chosen to be $S = (K^\C)^0$, then we  {write $\Lambda_{\mathcal{C}}^{+} G^{\mathbb{C}}_{\sigma} $.}
If the group $S$ is chosen to be $S = \{e\}$, then we write  {$\Lambda_{\star}^{+} G^{\mathbb{C}}_{\sigma} $.}

  {In some cases subgroups $S$, different from the above, are chosen, like in cases where there exists a Borel subgroup. In such cases one can derive a unique decomposition of any loop group element. In other cases, like in \cite{DoWa12}, one can only derive unique decompositions
for elements in some open subset of the given loop group:
\begin{remark}
If $G = SO^+(1,n+3)$ and $K = SO^+(1,3)\times SO(n)$, then there exists a closed, connected solvable subgroup $S \subseteq (K^\C)^0$ such that
the multiplication $\Lambda G_{\sigma}^0 \times \Lambda^{+}_S G^{\mathbb{C}}_{\sigma}\rightarrow
{(\Lambda G^{\mathbb{C}}_{\sigma})^0}$ is a real analytic diffeomorphism onto the open subset
$ \Lambda G_{\sigma}^0 \cdot \Lambda^{+}_S G^{\mathbb{C}}_{\sigma}     \subset(\Lambda G^{\mathbb{C}}_{\sigma})^0$.
\end{remark}
}

We frequently use the following decomposition theorems
(see \cite{Ke1}, \cite{DPW}, \cite{PS}, \cite{DoWa12}).

\begin{theorem} \label{thm-decomposition}\
\begin{enumerate}
\item {\em (Iwasawa decomposition)}
\begin{enumerate}
\item
$ (\Lambda G^{\C})_{\sigma} ^0=
\bigcup_{\delta \in \Xi }( \Lambda G)_{\sigma}^0\cdot \delta\cdot
\Lambda^{+}_\mathcal{C} G^{\mathbb{C}}_{\sigma},$
where $\Xi $ denotes {a (discrete) }set of representatives for the  double-coset decomposition;
\item The multiplication $\Lambda G_{\sigma}^0 \times \Lambda_\mathcal{C}^{+} G^{\mathbb{C}}_{\sigma}\rightarrow
(\Lambda G^{\mathbb{C}}_{\sigma})^0$ is a real analytic map onto the connected open subset
$ \Lambda G_{\sigma}^0 \cdot \Lambda_\mathcal{C}^{+} G^{\mathbb{C}}_{\sigma}   = {\mathcal{I}_e} \subset \Lambda G^{\mathbb{C}}_{\sigma}$.

 {Here $\mathcal{I}_e$ denotes the (connected) open Iwasawa cell containing the identity element.}
\end{enumerate}
\item  {\em (Birkhoff decomposition)}

\begin{enumerate}
\item
 {$(\Lambda {G}^\C )^0= \bigcup _{\omega \in \mathcal{W}} \Lambda^{-}_{\mathcal{C}} {G}^{\mathbb{C}}_{\sigma} \cdot \omega \cdot \Lambda^{+}_{\mathcal{C}} {G}^{\mathbb{C}}_{\sigma}$
where  $\mathcal{W}$ denotes a (discrete) set of representatives for the double coset
decomposition}

\item The multiplication $\Lambda_{*}^{-} {G}^{\mathbb{C}}_{\sigma}\times
\Lambda^{+}_\FC {G}^{\mathbb{C}}_{\sigma}\rightarrow
\Lambda {G}^{\mathbb{C}}_{\sigma}$ is an analytic  diffeomorphism onto the
open and dense subset $\Lambda_{*}^{-} {G}^{\mathbb{C}}_{\sigma}\cdot
\Lambda^{+}_\FC {G}^{\mathbb{C}}_{\sigma}$ {\em ( big Birkhoff cell )}.
\end{enumerate}
\end{enumerate}
\end{theorem}

\begin{remark} \ \label{S2}
\begin{enumerate}
\item  {We would  like to recall that for inner symmetric spaces  the twisted loop algebras
are isomorphic to the untwisted ones.
For the algebraic case see, e.g. \cite{Kac}, chapter 8, and our case follows by completion in the topology used in  this paper.}
\item  {The middle terms of the Birkhoff decomposition, item $(2)(a)$ above, form, in the untwisted case, the Weyl group
of the corresponding untwisted loop algebra (see e.g. \cite{PS}) and therefore form a discrete subset of the loop group.}
\item  {The middle terms in item $(1)(a)$ above can be determined quite precisely by using \cite{Ke1}.
Roughly speaking, they correspond to factors in a natural product decomposition of the Weyl group elements of the Birkhoff decomposition and thus they form a discrete subset of the loop group as well. For the present paper we will not need any special information about these factors.}
\item  {It is well known that in the Birkhoff decomposition only one of the double cosets is an open subset of the loop group under consideration. Therefore the name ``open cell" seems to be appropriate.
In the case of the Iwasawa decomposition, in general, several open double cosets can occur
(see, e.g., \cite{Ke1} for an explicit example and also \cite{D:open cells}).
But the open cell $\mathcal{I}_e$ containing the identity element plays naturally a special role.
Therefore it gets a name.}
\end{enumerate}
\end{remark}

Loops which have a finite Fourier expansion are called {\it algebraic loops} and
 denoted by the subscript $``alg"$, like
$\Lambda_{alg} G_{\sigma},\ \Lambda_{alg} G^{\mathbb{C}}_{\sigma},\
\Omega_{alg} G_{\sigma} $ as in  \cite{BuGu}, \cite{Gu2002}. And we define
  \begin{equation}\label{eq-alg-loop}\Omega^k_{alg} G_{\sigma}:
  =\left\{\gamma\in
\Omega_{alg} G_{\sigma}|
Ad(\gamma)=\sum_{|j|\leq k}\lambda^jT_j \right\}\subset \Omega_{alg} G_{\sigma} .\end{equation}


\subsection{ \bf{The DPW Method and its Potentials}}

 With the loop group decompositions as stated above, we obtain a
construction scheme of harmonic maps from a surface into $G/K$.

\begin{theorem}\label{thm-DPW} \cite{DPW}, \cite{DoWa12}, \cite{Wu}.
Let $\D$ be a contractible open subset of $\C$ and $z_0 \in \D$ a base point.
Let $\mathcal{F}: \D \rightarrow G/K$ be a harmonic map with $\mathcal{F}(z_0)=eK.$
The associated family $\mathcal{F}_{\lambda}$ (See Definition \ref{def-1}) of $\mathcal F$  can be lifted to a map
$F:\D \rightarrow \Lambda G_{\sigma}$, the extended frame of $\mathcal{F}$, and we can assume  without loss of generality that $F(z_0,\bar z_0, \lambda)= e$ holds.
Under this assumption,
\begin{enumerate}
\item
 The map $F$ takes only values in
 {$\mathcal{I}_e\subset \Lambda G^{\mathbb{C}}_{\sigma}$, i.e. in the open Iwasawa cell containing the identity element.}

 \item There exists a discrete subset $\D_0\subset \D$ such that on $\D\setminus \D_0$
we have the decomposition
\[F(z,\bar{z},\lambda)=F_-(z,\lambda)  F_+(z,\bar{z},\lambda),\]
where \[F_-(z,\lambda)\in\Lambda_{*}^{-} G^{\mathbb{C}}_{\sigma}
\hspace{2mm} \mbox{and} \hspace{2mm} F_+(z,\bar{z},\lambda)\in \Lambda^{+}_{\FC} G^{\mathbb{C}}_{\sigma}.\]
and $F_-(z,\lambda)$ is meromorphic in $z \in \D$  and satisfies
$F_-(z_0,\lambda) = e$.

Moreover,
\[\eta= F_-(z,\lambda)^{-1} \mathrm{d} F_-(z,\lambda)\]
 {is a $\lambda^{-1}\cdot\mathfrak{p}^{\mathbb{C}} \textendash \hbox{valued}$ meromorphic $(1,0) \textendash$form} with poles at points of $\D_0$ only.

\item Spelling out the converse procedure in detail we obtain:
Let $\eta$ be a  { $\lambda^{-1}\cdot\mathfrak{p}^{\mathbb{C}} \textendash \hbox{valued}$ meromorphic $(1,0) \textendash$form}  such that the solution
to the ODE
\begin{equation}
F_-(z,\lambda)^{-1} \mathrm{d} F_-(z,\lambda)=\eta, \hspace{5mm} F_-(z_0,\lambda)=e,
\end{equation}
is meromorphic on $\D$, with  $\D_0$ as set of possible poles.
 Then on  $\D_{\mathcal{I}}
= \{  z \in \D\setminus {\D_0}\ |\  {F_-(z,\lambda) \in \mathcal{I}_e}\}$
we
define  $\tilde{F}(z,\lambda)$ by the Iwasawa decomposition
\begin {equation}\label{Iwa}
F_-(z,\lambda)=\tilde{F}(z,\bar{z},\lambda)  \tilde{F}_+(z,\bar{z},\lambda)^{-1}.
\end{equation}
 This way one obtains  an extended frame
\[\tilde{F}(z,\bar{z},\lambda)=F_-(z,\lambda)  \tilde{F}_+(z,\bar{z},\lambda)\]
of some harmonic map from $  \D_{\mathcal{I}}  $ to $G/K$  satisfying
$\tilde{F}(z_0,\bar{z}_0,\lambda)= e$.

\item Any harmonic map  $\mathcal{F}: \D\rightarrow G/K$ can be derived from a
 {$\lambda^{-1}\cdot\mathfrak{p}^{\mathbb{C}} \textendash \hbox{valued}$ meromorphic
$(1,0) \textendash$form} $\eta$ on $\D$.
Moreover, the two constructions outlined above  are inverse to each other (on appropriate domains of definition).
\end{enumerate}
\end{theorem}

\begin{remark}\ \label{sphere}
\begin{enumerate}
\item   {A typical application of the theorem above arises as follows: one considers a harmonic map $\mathcal{F}: M \rightarrow G/K$, where $M$ is any Riemann surface and $G/K$ any inner semi-simple symmetric space
and considers the natural lift $\tilde{\mathcal{F}}: \tilde{M} \rightarrow G/K$. Then, if $M$ is compact of
positive genus or non-compact, then $\tilde{M}$ is contractible and one can apply the theorem above to  $\tilde{\mathcal{F}}$.}

\item  {So the question is: what happens if $M = S^2$? This case has been discussed in detail in Section 3.2 of \cite{DoWa12}. Basically, the theorem above still holds, if one admits  some
singular points. More precisely, if $\mathcal{F}: S^2 \rightarrow G/K$ is harmonic, then  {(see e.g. loc.cit. Theorem 3.11)} after removing at most two  {(different, but otherwise arbitrary)} points $\{p_1,p_2\}$ from $S^2$ one can find an extended frame
$F^\prime : S^2 \setminus{ \{p_1,p_2\}} \rightarrow  \Lambda G_\sigma$ for
$\mathcal{F}^\prime  = \mathcal{F}| S^2 \setminus{ \{p_1,p_2\}}  : S^2 \setminus{ \{p_1,p_2\}}  \rightarrow G/K$. Moreover, $F^\prime_-$, formed by Birkhoff decomposing $F^\prime$, extends  meromorphically  to $S^2$ and the normalized potential formed with $F^\prime_-$ extends meromorphically to $S^2.$
The converse{, the construction of a harmonic map defined on $S^2$ from a normalized potential,} can be carried out  {as usual} if one admits at most two singularities  {in the extended frame associated with the original normalized potential. For more details we refer to loc.cit.}
Of course, if one wants to obtain a harmonic map defined  {and smooth} on all of $S^2$, then additional conditions at the  {poles of the original normalized potential} need to be imposed.
We will use this result at several places below. }
\item
   The restriction above to factorizations   on  $\D_{\mathcal{I}}$ implies that on this set we have globally an Iwasawa decomposition of the form $(2.6)$ with $F_{\pm}$  globally smooth. This implies, of course, the smoothness of the associated harmonic map. At isolated points of $\D_{\mathcal{I}}$ in $\D$ the frames generally exhibit singular behaviour.
 In some cases, however, the corresponding harmonic maps are nevertheless non-singular.
 When considering, e.g.,  Willmore surfaces, singularities in the frame may or may not induce singularities in the associated harmonic map \cite{DoWa12}, \cite{Wang-3}. For example, singularities can occur in the extended frame, while both the associated harmonic conformal Gauss map as well as the corresponding Willmore surface stays smooth at the singularities \cite{Wang-3}.  Therefore, when discussing concrete examples of global immersions one needs to determine separately for all singularities of the frame, whether the final surface has a singularity, i.e. a branch point, or whether it is smooth and an immersion there. We refer to \cite{Wang-3} for examples of Willmore surfaces {with or without} branch points.
\end{enumerate}

\end{remark}

\begin{definition}\cite{DPW},\ \cite{Wu}.
 {With the conventions as above, let  $\mathcal{F}: M \rightarrow G/K$ be a
harmonic map with basepoint $p_0$, $\tilde{\pi}: \tilde{M} \rightarrow M$ the universal cover of $M$
and  $\tilde{\mathcal{F}}: \tilde{M} \rightarrow G/K$
the natural lift of $\mathcal{F}$ with basepoint $z_0,$ where $\tilde{\pi}(z_0) = p_0$.
Then the
$\lambda^{-1}\cdot \mathfrak{p}^{\mathbb{C}} \textendash \hbox{valued}$
 meromorphic $(1,0)\textendash$form  $\eta$  defined in $(2)$ of the last theorem for
 $\tilde{\mathcal{F}}$ is called the {\em normalized potential} for the harmonic
map $\mathcal{F}$ with the point $z_0$ as the reference point. And $F_-(z,\lambda)$, satisfying
 $F_-(z_0,\lambda) = e,$ given above is called the  {corresponding} meromorphic extended frame.}
\end{definition}

The normalized potential is uniquely determined,  {since the extended frames are normalized to $e$ at some fixed base point on $\tilde{M}$.}
The normalized potential is  meromorphic  {in $z \in \tilde{M}$.}

In many applications it is much more convenient to use potentials which have a Fourier expansion containing more than one power of $\lambda$.
And when permitting many (maybe infinitely many) powers of $\lambda$,  one can
 {obtain holomorphic coefficients:}

\begin{theorem}\cite{DPW}, \cite{DoWa12}.\label{thm-CC}
Let $\D$ be a contractible open subset of $\C$.
\begin{enumerate}
    \item Let $F(z,\bar{z},\lambda)$ be an extended  frame of some harmonic map from  $\D$  to $G/K$. Then there exists  {some real-analytic $V_+: \D  \rightarrow  \Lambda^{+} G^{\mathbb{C}}_{\sigma} $ such that $C(z,\lambda) =
F(z, \bar z, \lambda)  V_+ (z, \bar z,\lambda) $} is holomorphic in $z\in\mathbb{D}$ and in $\lambda \in \mathbb{C}^*$, with $C(z_0,\lambda) =e$ and $V_+(z_0, \bar{z}_0, \lambda) = e$. The Maurer-Cartan form $\eta = C^{-1} \mathrm{d} C$ of $C$ is a holomorphic
$(1,0) \textendash$form on $\D$ and  $\lambda \eta$ is holomorphic for $\lambda \in \C$.

\item
Conversely, Let $\eta\in\Lambda\mathfrak{g}^{\C}_{\sigma}$ be a holomorphic
$(1,0)\textendash$form such that $\lambda \eta$ is holomorphic  in $\lambda$ for $\lambda \in \C$, then by the same process  {as} given in Theorem \ref{thm-DPW} we obtain a harmonic map $\mathcal{F}: \D \rightarrow G/K$.
\end{enumerate}

\end{theorem}

 {\begin{definition}\label{rm-C}
 The matrix function $C(z,\lambda)$ associated with the holomorphic $(1,0) \textendash$form $\eta$ as in Theorem \ref{thm-CC} will be called a {\em holomorphic extended frame} for the harmonic map $\mathcal{F}$.
 \end{definition}}

\subsection{\bf{Symmetries and Monodromy}}

 {It is natural to investigate harmonic maps with symmetries.}
Since harmonic maps frequently occur as ``Gauss maps" of some surfaces,
 {the
investigation  of harmonic maps with symmetries also has  implications for surface theory.
}

\begin{definition}
 {Let  $\mathcal{F}: M \rightarrow G/K$ be a harmonic map. Then a pair, $(\gamma,R)$, is called a symmetry of $\mathcal{F}$, if
$\gamma$ is an automorphism of $M$ and R is an automorphism of $G/K$ satisfying
 $$\mathcal{F}(\gamma.p) = R.\mathcal{F}(p)$$
  for all $p \in M$.}
  \end{definition}

   {We would like to point out that an intuitive notion of "symmetry" for $\mathcal{F}$ would be an automorphism $R$ of $G/K$ such that  $ R \mathcal{F}(M) = \mathcal{F}(M)$. In some cases one can prove that this intuitive definition implies the actual definition given just above.}

  \begin{lemma} \label{frametransform}
 {  Let  $\mathcal{F}: M \rightarrow G/K$ be a harmonic map {and}  $(\gamma,R)$ a symmetry
  of $\mathcal{F}.$
Let $\tilde{M}$ denote the universal cover of $M$
and $\tilde{\mathcal{F}}: \tilde{M} \rightarrow G/K, \tilde{\mathcal{F}} = \mathcal{F} \circ \pi$, its natural lift. Then:
 \begin{enumerate}
 \item  $\tilde{\mathcal{F}}$ satisfies
\[\tilde{\mathcal{F}} (\gamma.z) = R.\tilde{\mathcal{F}}(z).\]
\item
 For any frame $F : \tilde{M} \rightarrow G$  of $\mathcal{F}$ one  obtains
\begin{equation} \label{symmetry-nolambda}
 {\gamma^*F(z,\bar z) = RF(z,\bar z)k(z,\bar z),}
\end{equation}
where $k(z,\bar z)$ is a function from $\tilde{M}$ into $K$.
\item  For the extended frame  {$F(z,\bar z,\lambda) : \tilde{M} \rightarrow \Lambda G_\sigma$
of $\mathcal{F}$} there exists some map
 $\rho_\gamma: \C^* \rightarrow \Lambda G_\sigma$ such that
 \begin{equation} \label{symmetry}
\gamma^*F(z,\bar z,\lambda)  = \rho_\gamma (\lambda) F(z,\bar z,\lambda)  {k(z, \bar z)},
\end{equation}
where $k$ is the $\lambda \textendash$independent function from $\tilde{M}$ into $K$ occurring in the previous equation.
Moreover, $\rho_{\gamma}(\lambda)|_{ \lambda= 1} = R$ holds.
  \end{enumerate}}
\end{lemma}

 {Note, since $\mathcal{F}$ is full, for each symmetry $(\gamma,R)$ the automorphism $R$ of $G/K$
is uniquely determined by $\gamma$. We therefore write $\rho_\gamma (\lambda) = \rho(\gamma,\lambda)$ and ignore $R$ in this notation. Also note that $\rho_\gamma$ actually is defined and holomorphic for all $\lambda \in \C^*$.}

\begin{proof}
 The first two equations follow immediately from the definitions. In view of (\ref{alphalambda})
the equality of Maurer-Cartan forms of (\ref{symmetry-nolambda}) implies the equality of the Maurer-Cartan forms of (\ref{symmetry}). Therefore only the last statement needs to be proven. But evaluating  (\ref{symmetry}) at the base point $z_0$ of $\tilde{\mathcal{F}}$ yields
 $F(\gamma.z_0,\overline{\gamma.z_0},\lambda)  = \rho_\gamma (\lambda) k (z_0,\bar z_0).$
 Hence we obtain
$\rho_\gamma: S^1 \rightarrow \Lambda G_\sigma,$ with holomorphic extension to $\C^*$, and putting $\lambda = 1$  we infer
 $F(\gamma.z_0,\overline{\gamma.z_0}) = \rho_\gamma k(z_0,\bar z_0).$ On the other hand, from
 (\ref{symmetry-nolambda})
we obtain  $F(\gamma.z_0, \overline{\gamma.z_0}) = R k(z_0,\bar{z}_0)$, whence
 $R =\rho_{\gamma}(\lambda)|_{ \lambda= 1} $.
\end{proof}

\begin{definition}
 {With the notation above,  the matrix $\rho_\gamma (\lambda), \lambda \in S^1,$ ( for all $\lambda\in\C^* $ in fact) occurring in (\ref{symmetry})  is called the
monodromy (loop)  matrix of $\gamma$  for $\mathcal{F}.$}
\end{definition}

The following result has been proven in Theorem 4.8 of  {\cite{Do-Wa-sym}.}

\begin{theorem}
Let $M$ be a Riemann surface which is either non-compact or compact of positive genus.
\begin{enumerate}
\item
 Let $\mathcal{F}:M \rightarrow G/K$ be a harmonic map and
  {$\tilde{\mathcal{F}}: \tilde{M} \rightarrow G/K$  its natural lift to the universal cover $\tilde{M}$ of $M$.} Then there exists a normalized potential and a holomorphic potential for $\mathcal{F}$, namely the corresponding  {potential} for $\tilde{\mathcal{F}}$.

\item Conversely, starting from some potential producing a harmonic map
$\tilde{\mathcal{F}}$ from $\tilde{M}$ to $G/K$, one obtains a harmonic map $\mathcal{F}$ on $M$ if and only if

\begin{enumerate}
\item  The monodromy matrices $\chi(g, \lambda)$ associated with
$g \in \pi_1 (M)$, considered as automorphisms of $\tilde{M}$, are elements of $(\Lambda G_{\sigma})^0$.

\item  There exists some $\lambda_0 \in S^1$ such that  {$\chi(g, \lambda)|_ { \lambda= \lambda_0} =e$, i.e.
\begin{equation*}
\begin{split} F(g.z,\overline{g.z},\lambda)|_{ \lambda= \lambda_0}&\equiv
\chi(g, \lambda)|_{ \lambda= \lambda_0} F(z, \bar{z}, \lambda)|_{ \lambda= \lambda_0}\ \mod\ K\\
&\equiv   F(z, \bar{z}, \lambda)|_{ \lambda= \lambda_0}\ \mod\ K
\end{split}
\end{equation*}}
for all $g \in \pi_1 (M)$.

\end{enumerate}

\end{enumerate}
\end{theorem}

 {We also need the existence of normalized potentials for harmonic maps from a $2-$sphere
to an inner symmetric space, compact or non-compact.  An important difference to the previously discussed cases is that the extended frames will not be smooth globally on $S^2$,  but will be smooth (actually real analytic) on
$S^2 \setminus {\hbox{\{two points\}}}.$  See Remark \ref{sphere}.
The details can be found in $(2)$ of Remark \ref{sphere}
or Theorem 3.11 of \cite{DoWa12} which  we recall for the convenience of the reader:}

\begin{theorem}\label{normalized-potential-sphere} (Theorem 3.11 of  \cite{DoWa12})
Every  harmonic map from $S^2$ to any  Riemannian or pseudo-Riemannian
  symmetric space  $G/K$ admits an extended frame with at most two  singularities.
 Furthermore,  it admits a global meromorphic extended frame. In particular, every  spacelike conformal  harmonic map from $S^2$ to any  Riemannian or pseudo-Riemannian  symmetric space  $G/K$ can be obtained from some meromorphic normalized potential.
\end{theorem}

 {While we generally restrict in this paper to inner symmetric spaces, the last and the following theorem are stated more generally because of their importance.  As stated before, the case of outer symmetric spaces will not be considered in any detail in this paper.}

\subsection{\bf{Invariant Potentials}}

 { If $M$ has a non-trivial fundamental group, then  invariant potentials are  of particular interest. The first part of the theorem below (the case of a non-compact $M$) has been proven first  in \cite{Do-Ha2} for the case of  harmonic maps into $S^2$.   The case of harmonic maps from a compact Riemann surface  $M$
 into the (inner) symmetric space $G/K=SO^+(1,n+3)/SO^+(1,3)\times SO(n)$ was treated in
 Section 6 of  \cite{Do-Wa-sym}. }

  {For the case of compact $M$  no proof for a general inner  (semi simple) symmetric space  is known.
 For the case of non-compact $M$  the published proofs are not very clear nor explicit.
 We therefore include a new proof using work of Bungart \cite{Bungart} and R\"ohrl \cite{Roehrl}.}

\begin{theorem} \label{thm-monodr}
Let $M$ be a Riemann surface.

\begin{enumerate}

 \item If $M = S^2$, then every harmonic map from $S^2$ to any symmetric space can be generated from some meromorphic (clearly invariant)  potential on $S^2$.

\item If $M$ is compact, but different from $S^2$, then every harmonic map from $M$ to
 {the inner symmetric space  $G/K = SO^+(1,n+3)/SO^+(1,3)\times SO(n)$}  can be generated by some meromorphic  potential defined on $M$, i.e. from  a  meromorphic potential defined on the (contractible) universal cover $\tilde{M}$ of $M$ which is invariant under the fundamental group of $M$.

\item If $M$ is non-compact, then every harmonic map from $M$ to any symmetric space  {$G/K$} can be generated from some invariant holomorphic potential on $M$, i.e. it can be generated from some holomorphic potential on the universal cover $\tilde{M}$ of $M$ which is invariant under the fundamental group of $M$.

\end{enumerate}
\end{theorem}

\begin{proof}
(1) comes from $(2)$  of Remark \ref{sphere} and the reference given there.

 (2) comes from Section 6 of \cite{Do-Wa-sym}.

 (3) In this case $M$ is non-compact and $\mathcal{F} : M \rightarrow G/K$ is a harmonic map into an inner  symmetric space $G/K$.  Let $F(z,\bar z,\lambda)$ be an extended frame
of  $\mathcal{F}.$
Then by  \cite[Lemma 4.11]{DPW} we obtain a  $($real analytic$)$ matrix function
 $\tilde V_+: \tilde{ M} \rightarrow \Lambda^+ G^\C_\sigma$ such that the matrix
 $C$ defined by
\begin{equation}
 C(z,\lambda) := F(z,\bar z,\lambda) \tilde{V}_+(z,\bar z, \lambda)
\end{equation}
is holomorphic in $z \in \tilde{M}$ and $\lambda \in \C^*$.
 Moreover, $F(z,\bar z,\lambda)$ satisfies  (\ref{symmetry}).
 As a consequence,  $C$ inherits
 from $F(z,\bar z,\lambda)$ the transformation behaviour
  \begin{equation}
  C(\tau.z, \lambda) = M(\tau, \lambda)  C(z,\lambda) W_+(\tau, z, \lambda),
 \end{equation}
where $\tau \in  \pi_1(M)$  and  $W_+: \tilde{ M} \rightarrow  \Lambda^+ {G^\C}_\sigma$
 is holomorphic in $z$ and $\lambda \in \C^*$.  It is straightforward to verify the ``cocycle condition"
 \begin{equation}
W_+(\tau \mu,z,\lambda) =
W_+(\tau, \mu.z,\lambda) W_+(\mu,z,\lambda)  \hbox{ for all  $\tau, \mu \in \pi (M ).$}
\end{equation}
Our goal is to split the cocycle $W_+(\tau,z,\lambda)$ in $\Lambda^+ G^\C_\sigma$. For this we begin by following  the first few lines of the proof of Theorem 3 of  \cite{Roehrl}.The paper refers to complex Lie groups,
  which, in our case we consider the complex Banach Lie group $H = \Lambda^+ G^\C_\sigma$.
   If $H$ is a complex Banach Lie group, then we denote by
 ${(H_\omega})^\mathcal{C}$  the sheaf of  continuous sections
from open subsets of $M$ to $H$.
Similarly, by  ${(H_\omega})^\mathcal{H}$ we denote the sheaf of  holomorphic sections
from open subsets of $M$ to $H$.

  First we prove:
\begin{itemize}
\item[(a)]  Let $M$ be a non-compact Riemann surface and $H^\C$ a complex Banach Lie group
then $H^1(M, {(H_\omega)}^\mathcal{C})$ = 0.
\end{itemize}

For the convenience of the reader we translate the first 10 or so
lines of this proof:
For $\xi\in H^1(M,H)$  we need to prove that in the principal bundle associated
with $\xi$ there exists a continuous section. Since the principal bundle can contain, for dimension reasons, at most two-dimensional obstructions, it suffices for the existence of a continuous section the verification that the two-dimensional obstruction vanishes. But this obstruction is an element of $H^2(M, \pi_1(H))$. Moreover, for a non-compact Riemann surface $M$ it is known that $H_2(M,\mathbb{Z})$ vanishes, whence by the universal coefficient theorem also $H^2(M, \pi_1(H))$ vanishes. This proves claim \textrm{(a)}.
 From this we derive the complex Banach group version of
 \cite[Theorem 3]{Roehrl}:
\begin{itemize}
\item[(b)]
 Let $M$ be a non-compact Riemann surface and $H$ a complex Banach Lie group
then $H^1(M, {(H_\omega)}^\mathcal{H}) = 0.$
\end{itemize}
 {But \cite[Theorem 8.1]{Bungart} implies}
\[
 H^1(X, {(G^\C_\omega)}^\mathcal{C}) \cong H^1(X, {(G^\C_\omega)}^\mathcal{H}),
\]
and claim $(b)$ follows.

 To finish the proof of the splitting  result  for the cocycle  $W_+(\tau, z, \lambda),$ we can now use
 \cite[Exercise 31.1]{Forster}.
 For a detailed proof one can follow the proof of
 \cite[Theorem 31.2]{Forster}, but with
 $\Psi_i(z)=W_+(\eta_i(z)^{-1},z,
 \lambda)^{-1}$ where we refer for notation to loc.cit.

 We now know $ W_+(\tau, z, \lambda) = P_+(z,\lambda)  P_+(\tau.z,\lambda)^{-1} $ and consider
$\hat{C} =  C P_+$.
A simple computation shows
$ \hat{C}(\tau.z,\lambda) = M(\tau, \lambda) C(z, \lambda)$
  for all $\tau \in \pi_1(M)$ and all
 $\lambda \in \C^*$.  As a consequence, the differential one-form $ \eta = \hat{C}^{-1} \dd\hat{C}$
 is invariant under $\pi_1(M)$.
  \end{proof}

 {The differential $1$-form $ \eta = \hat{C}^{-1} \dd\hat{C}$ as   above  will be called
 an \emph{invariant holomorphic potential} for the   harmonic map $\mathcal{F}$.


\section{\bf{Algebraic Harmonic maps and Totally Symmetric Harmonic Maps}}



In this section we denote by $\tilde{M}$ a simply-connected Riemann surface, i.e., $S^2$,  $\E=\{z\in\C\ | \ |z|<1 \}$,
or $\C$. Let $G/K$ denote an inner symmetric space and $\mathcal{F}:\tilde{M}\rightarrow G/K$ a harmonic map. Here $G$ is a semi-simple Lie group with trivial center.
 We would like to recall more conventions:
{Let $\mathcal{F}:\tilde{M} \rightarrow G/K$ be a harmonic map defined on a simply-connected Riemann
surface $\tilde{M}$. Then we assume  {at least tacitly} that a basepoint  $z_0\in \tilde {M}$  is chosen and
we also assume that any extended frame $F$, the corresponding normalized  extended frame $F_-$  and  {any} holomorphic extended frame $C$  {associated with an extended frame $F$} all attain the  value $e$ at $z_0$.}
 {Some special feature for the case of $M = S^2$ has been addressed in item  {$(2)$} of Remark \ref{sphere}, including the reference given there.}

\subsection{\bf{Algebraic Harmonic Maps}}

We first consider the property $(U2)$ of Theorem \ref{Theorem1.1}.

\begin{definition} (Algebraic Harmonic Maps) Let $\tilde{M}$ be a simply-connected Riemann surface. A harmonic map $\mathcal{F}:\tilde{M}\rightarrow G/K$  {is said} to be of finite uniton type if some  extended frame $F$ of $\mathcal{F},$ i.e. some frame satisfying $F(z_0,\bar z_0, \lambda)=e$ for some (fixed) base point $z_0\in \tilde{M},$ is a Laurent polynomial in $\lambda$.
\end{definition}

 \begin{remark}
 \begin{enumerate}
     \item Note that the condition
 $F(z_0,\bar z_0,\lambda) = e$ is not a restriction. Assume some $\hat{F}$
satisfies the conditions  of the definition above, except $\hat{F}(z_0,\bar z_0,\lambda) = e$,
then $F(z,\bar z,\lambda) = \hat{F}(z_0,\bar z_0,\lambda)^{-1} \hat{F}(z,\bar z,\lambda)$ satisfies all conditions.
\item
We would also like to point out that for any extended frame $F$ of some harmonic map
which is a Laurent polynomial in $\lambda$, both factors in the unique meromorphic Birkhoff decomposition  $F(z,\bar z,\lambda) = F_-(z,\lambda) F_+(z,\bar z,\lambda)$, i.e. assuming
$F_-(z,\lambda) = e + \mathcal{O}(\lambda^{-1})$, also are Laurent polynomials.
\item If a harmonic map is algebraic, then any extended frame is a Laurent polynomial in $\lambda$.
 \end{enumerate}

\end{remark}

\begin{proposition}\label{prop-fut}   Let $\mathcal{F}:\tilde{M} \rightarrow G/K$ be a harmonic map defined on a  {contractible} Riemann surface $\tilde{M}$.
 { Let $z_0\in \tilde {M}$ be a base point. Then the following statements are equivalent:}
\begin{enumerate}
\item   $\mathcal{F}$ is  an algebraic harmonic map.
\item There exists a  frame  {$F(z,\bar z,\lambda)$} of $\mathcal{F}$ which is a Laurent polynomial in $\lambda$.
\item  The normalized extended frame  {$\hat{F}_-(z,\lambda)$ of any extended frame
$\hat{F}(z,\bar z,\lambda)$ of $\mathcal{F}$} is a Laurent polynomial in $\lambda$.
\item  Every holomorphic  {extended frame $\hat{C}(z,\lambda)$}  associated to any extended frame
 {$\hat{F}(z,\bar z,\lambda)$} of $\mathcal{F},$ thus satisfying $\hat{C}(z_0, \lambda) = e$ for all $\lambda,$
 only  contains finitely many negative powers of $\lambda$.
\item There exists a holomorphic extended  { frame $C^\sharp(z,\lambda)$ } which  { only contains} finitely many negative powers of $\lambda$.
\end{enumerate}
    For the case $\tilde{M} = S^2,$  the above   {equivalences remain} true, if one replaces in the last two statements the word ``holomorphic" by ``meromorphic" and  {also admits for all
    frames  at most two singularities.}
\end{proposition}

\begin{proof}
 For a contractible domain, the definition of an algebraic harmonic map  can be
rephrased by $(1) \Leftrightarrow (2)$. Hence we only need to show that  $(2), (3), (4)$ and $(5)$ are equivalent.

(2) $\Rightarrow $  (3):  {Let $F(z,\bar z,\lambda)$ and $\hat{F}(z,\bar z,\lambda)$ be as in $(2)$ and $(3)$ respectively,
w.l.g. being extended frames. Then
$\hat{F} (z,\bar z,\lambda)= F (z,\bar z,\lambda)k(z,\bar z,\lambda)$. Inserting the unique Birkhoff decompositions $F(z,\bar z,\lambda) = F_-(z,\lambda)F_+(z,\bar z,\lambda)$ and $\hat{F}(z,\bar z,\lambda) = \hat{F}_- (z,\lambda)\hat{F}_+(z,\bar z,\lambda)$  we infer $F_-(z,\lambda) = \hat{F}_-(z,\lambda)$ and the claim follows.}

(3) $\Rightarrow$ (4):  {By Theorem \ref{thm-CC},}  { $ \hat{C}(z,\lambda) = \hat{F} (z,\bar z,\lambda)\hat{V}_+(z,\bar z,\lambda)$, where
$\hat{V}_+ $ is actually real-analytic.
 Inserting the   Birkhoff decomposition $\hat{F}(z,\bar z,\lambda) = \hat{F}_-(z,\lambda) \hat{F}_+(z,\bar z,\lambda)$, we obtain
 $\hat{C} (z,\lambda)= \hat{F}_-(z,\lambda) \hat{F}_+(z,\bar z,\lambda) V_+(z,\bar z,\lambda) $ and the claim follows.}

(4) $\Rightarrow$ (2):  {Let $\hat{C}(z,\lambda)$ be any holomorphic extended frame for $\mathcal{F}$. Consider the Iwasawa decomposition  $F(z,\bar z,\lambda)=\hat{C}(z,\lambda)  \tilde{C}_+(z,\bar z,\lambda)$ near $z_0$, where  $\tilde{C}_+\in \Lambda^{+} G^{\mathbb{C}}_{\sigma}$
and where $F(z,\bar z,\lambda)$ and  $\tilde{C}_+(z,\bar z,\lambda)$ attain the  value $e$ at $z_0$.
Since $\hat{C}(z,\lambda) $ only contains finitely many negative powers of $\lambda$ by assumption, also $F(z,\bar z,\lambda)$ contains only finitely many negative powers of $\lambda$. Since $F(z,\bar z,\lambda)$ is real, it is a Laurent polynomial.}

(4) $ \Leftrightarrow$ (5): Note that for any two holomorphic extended frames $C_1(z,\lambda)$ and $C_2(z,\lambda)$ there exists some $W_+(z,\lambda)\in \Lambda^{+}G^{\mathbb{C}}_{\sigma}$ such that $C_1(z,\lambda)=C_2(z,\lambda)W_+(z,\lambda)$ holds.

 For the case $\tilde{M}=S^2$,   {using the results just proven} for $M_1=S^2\setminus\{\infty\}$ and $M_2=S^2\setminus\{0\}$ respectively, one will obtain  {the last claim of the proposition. For more details see
 $(2)$  of Remark \ref{sphere} and Section 3.2 of \cite{DoWa12}.}
 \end{proof}

We wonder, under what conditions the
based (at $z_0$) normalized  extended frame  {$F_-(z,\lambda)$} will be a Laurent polynomial.
 Let $\eta$ denote the normalized potential  { $F_-(z,\lambda)^{-1} \mathrm{d} F_-(z,\lambda) = \eta$.
 Then $F_-(z,\lambda), F_-(z,\lambda)|_{z=z_0} = e$}, can be obtained from $\eta$ by an application  of the  {standard Picard iteration of the theory of ordinary differential equations. Since $\eta$ is a multiple of
 $\lambda^{-1},$} it is easy to see that each step of the Picard iteration
  {decreases the occurring power of $\lambda$ by $-1$.} So $F_-(z,\lambda)$ is a Laurent polynomial if and only if the Picard iteration stops after finitely many steps. { See for example, Section 1 of \cite{Gu2002}}.
The most natural reason for the  Picard iteration to stop is that the  normalized potential $\eta$ takes values in some nilpotent Lie algebra \footnote{Note: If $\eta(z)$ is only nilpotent for every $z\in \tilde M$, it does not follow, in general, that the Picard iteration will stop.}.

The following result is a slight generalization of Appendix B of \cite{BuGu}, (1.1) of \cite{Gu2002}. We would like to point out  that in particular $G$ does not need to be compact.

\begin{proposition}  Let $\mathfrak{n}\subset\mathfrak{g}^{\C}$ be a nilpotent subalgebra and assume that $\eta$ is a (holomorphic or meromorphic) potential of some harmonic map such that  $\eta(z)\in\mathfrak{n}$ for all
 $z\in \tilde M \setminus \lbrace$ poles of $\eta \rbrace $, and $\eta$ only contains finitely many positive powers of $\lambda$.  Let $\mathcal{F}:\tilde{M}_0\rightarrow G/K$ be the harmonic map associated with $\eta$ on an open subset $\tilde{M}_0\subset\tilde{M},$ $z_0 \in \tilde{M}_0$.  Then  $\mathcal{F}$ is algebraic..
\end{proposition}

\begin{proof}
When  $\eta(z)\in\mathfrak{n}$ for all $z\in \tilde M$, then it is easy  to see that the Picard iteration producing the solution
 {$\mathrm{d} C(z,\lambda)= C(z,\lambda) \eta$,} $C(z_0,\lambda)=e$, stops after finitely many steps.
Since $\eta$ is a Laurent polynomial in $\lambda$,  {so is the solution} $C$. By Proposition \ref{prop-fut}, $\mathcal{F}$ is algebraic.
\end{proof}

\begin{remark}
\
\begin{enumerate}
\item
The conformal Gauss maps of many Willmore surfaces (see \cite{DoWa11}, \cite{Wang-1} and \cite{Wang-3}) can be constructed as discussed in the proposition above.
In these examples  the group $G$ is non-compact.

 \item For harmonic maps into arbitrary Lie  groups $G$ with a bi--invariant non--degenerate metric, one can also produce harmonic maps of finite uniton type following the above procedure by considering $G$ as the symmetric space
$(G \times G)/G$. Harmonic $2-$spheres in $U(4)$
provide standard examples for the proposition (see Section 5 of \cite{BuGu} or Appendix B of \cite{Gu2002}). In this case the group $G = U(4)$ is compact.
\item
For the construction of examples of Willmore spheres it is  important to show that harmonic maps defined on $M=S^2$ are algebraic. We will show this in \cite{DoWa-fu2}.
\end{enumerate}
\end{remark}




We close this subsection by an application of the Duality Theorem in \cite{DoWa13}.
\begin{theorem}   \label{duality:algebraic}
We retain the general assumptions of this paper and assume in addition that $M$ is contractible.
Let $G/K$ be an inner  symmetric space of non-compact type and  $\tilde{U}/(\tilde{U}\cap \tilde{K}^{\C})$
its compact dual.
Let $\mathcal{F} : M \rightarrow G/K$ and
$\tilde{\mathcal{F}}_U : \tilde{U} \rightarrow   \tilde{U}/(\tilde{U}\cap \tilde{K}^{\C})$
be a ``dual pair" of harmonic maps as constructed in then Appendix, then
\[\hbox{ $\mathcal{F}$   is algebraic  $\Longleftrightarrow$  $\tilde{\mathcal{F}} $ is algebraic.}\]
\end{theorem}

\begin{proof}
This follows from the crucial Iwasawa decomposition formula
$$F(z,\bar{z},\lambda)=F_{\tilde{U}}(z,\bar{z},\lambda) S_+(z,\bar{z},\lambda)$$
which can be read both ways.
So if $F(z,\bar{z},\lambda)$ is an extended frame for $\mathcal{F}$ which is a Laurent polynomial in
$\lambda$, then
$F_{\tilde{U}}(z,\bar{z},\lambda) = F(z,\bar{z},\lambda) S_+(z,\bar{z},\lambda)^{-1}$  shows that also
$F_{\tilde{U}}(z,\bar{z},\lambda)$ only contains finitely many negative powers of $\lambda.$ Since
$F_{\tilde{U}}(z,\bar{z},\lambda)$ is real, it follows that $F_{\tilde{U}}(z,\bar{z},\lambda)$ actually is a
Laurent polynomial in $\lambda$. The converse follows analogously.
\end{proof}

For the case $M = S^2$,  we can perform Theorem \ref{duality:algebraic} to  $S^2 \setminus{ \{p_1\}}$ and $ S^2 \setminus{ \{p_2\}} $ respectively, in the same way as done in \cite[Section 3.2]{DoWa12}.

\subsection{ \bf{Totally Symmetric Harmonic maps}} \label{trivmonorep}

We retain the general assumptions of this paper, and of this section. In the last subsection we considered exclusively simply-connected Riemann surfaces $M$. Obviously then, the monodromy representation of any harmonic map $\mathcal{F}:M \rightarrow G/K$
is trivial. But there also exist non-simply-connected Riemann surfaces which admit harmonic maps $\mathcal{F}:M \rightarrow G/K$ with  trivial monodromy representation. It is this type of harmonic maps which we will discuss in this subsection.




\begin{definition}\label{def-uni} (Totally Symmetric Harmonic Maps) {We retain all general assumptions made for this paper. Let $M$ be a Riemann surface with universal cover
$\tilde{M}$ and let $G/K$ be an inner symmetric space. A harmonic map
$\mathcal{F}:M\rightarrow G/K$ is said to be \emph{totally symmetric,}  if there exists some extended frame $F(z,\bar{z},\lambda):\tilde M\rightarrow G$ of $\mathcal{F}$} which has trivial monodromy  for all $g \in \pi_1(M)$; i.e.
\[F(g.z,\overline{g.z},\lambda)=  F(z,\bar{z},\lambda) k(g,z,\bar{z}),\hbox{for all $g \in \pi_1(M)$ and all $\lambda \in S^1$ (resp. $\lambda \in \C ^*)$.}\]
This is equivalent to that
 for all $\lambda \in S^1,$ any extended frame $F(z,\bar{z},\lambda): \tilde{M} \rightarrow (\Lambda  {G _{\sigma})} $ descends to a well defined map on $M$, i.e.
    $$F(z,\bar{z},\lambda): M \rightarrow (\Lambda G_{\sigma})/K,$$
up to two singularities in the case of $M = S^2$.
 \end{definition}

 \begin{remark}\
  \begin{enumerate}
  \item
{We would also like to point out that we will use, by abuse of notation,  the same notation for the frame $F$ with values in
$\Lambda G _{\sigma}$ and its projection with values in $ \Lambda G_{\sigma}/K$.}
\item
Moreover, we sometimes express equivalently ``totally symmetric" by ``with trivial monodromy (representation)".
  \item Note that  in Section 4 below we discuss the relations between the DPW theory (using extended frames) and the Burstall-Guest theory (using extended solutions).  It turns out that  the properties ``algebraic" and ``totally symmetric"
 taken together are closely related to  the notion of  a minimal  (finite) uniton number
 harmonic  map defined on page 546 of \cite{BuGu}. ( Also see Proposition \ref{typeequivnumber}.)


\item Condition of being ``totally symmetric" is a very strong condition. Of course, in the case of a non-compact simply connected Riemann surface $M$   {and  {of} $S^2$ respectively,} it is always satisfied.
\end{enumerate}
\end{remark}


 Recall, that we assume that all harmonic maps in this paper are considered to be full,
 {see Definition \ref{deffull}.} In the following proposition we will use, as in the definition just above, by abuse of notation the same notation for the maps $C $ and $F_-$
and their projections respectively to some quotient space.}

\begin{proposition} \label{prop-frame}
  Let   {$M \neq S^2$}   be any Riemann surface and let $\mathcal{F}: M\rightarrow G/K$ be a harmonic map which is full in $G/K$.  Then the following statements are equivalent

   \begin{enumerate}
\item The  monodromy matrices $\chi(g,\lambda)$, $g\in\pi_1(M)$, $\lambda\in S^1$, satisfy $\chi(g,\lambda)=e,\ \hbox{ for all } g\in\pi_1(M),\ \lambda\in S^1$.

\item Any extended frame $F$ of $\mathcal{F}$ is defined
on $M$ ``modulo $K$'', i.e. $F$ descends from a map defined on $\tilde{M}$ with values in
${\Lambda G}_{\sigma}$ to
$F(z, \bar{z},\lambda ): M \rightarrow  ({\Lambda G}_{\sigma}) / K$.

\item Any holomorphic extended frame $C$ of $\mathcal{F}$ is defined on $M$
``modulo $\Lambda^+G^{\mathbb{C}}_{\sigma}$'' , i.e. $C$ descends from a map defined on $\tilde{M}$ with values in   $\Lambda G^{\C}_{\sigma}$ to
$C(z,\lambda):M\rightarrow \Lambda G^{\C}_{\sigma}/\Lambda^+ G^{\C}_{\sigma}$.

 \item The normalized   extended frame $F_-$ relative to any given base point $z_0$ is defined on $M$, i.e. $F_-$ descends from a map defined on $\tilde{M}$ with values in $\Lambda^- G^{\C}_{\sigma}$ to  $F_-(z,\lambda):M\rightarrow \Lambda^- G^{\C}_{\sigma}$.
 \end{enumerate}
 \end{proposition}

\begin{proof}
Let's consider any Riemann surface $M$ different from $S^2$.
By general theory we have for every $g\in\pi_1(M)$ and all $\lambda\in S^1$
on the universal cover $\tilde{M}$ the equations:
\[\begin{split}
F(g.z,\overline{g.z},\lambda)&= \chi(g,\lambda) F(z,\bar{z},\lambda) k(g,z,\bar{z}),\\
C(g.z,\lambda)&=\chi(g,\lambda) C(z,\lambda) W_+(g,z,\bar{z},\lambda),\\
F_-(g.z,\lambda)&=\chi(g,\lambda) F_-(z,\lambda) L_+(g,z,\bar{z},\lambda),\\
  \end{split}\]
for some maps
$k(g,z,\bar{z}): \tilde{M} \rightarrow K $,
$ W_+ (g,z,\bar{z},\lambda),
\ L_+(g,z,\bar{z},\lambda): \tilde{M}\rightarrow \Lambda^+ G^{\C}_{\sigma}$, and $\chi(g,\lambda)\in\Lambda G_{\sigma}$. Therefore (1) implies (2), (3) and (4).

``(2) $\Rightarrow$ (1)": From the assumption we obtain $F(g.z,\overline{g.z},\lambda)= F(z,\bar{z},\lambda)\tilde{k}(g,z,\bar{z})$ for some $\tilde{k} : \tilde{M} \rightarrow K$. This shows \[\chi(g,\lambda)F(z,\bar{z},\lambda)k(g,z,\bar{z})=F(z,\bar{z},\lambda)\tilde{k}(g, z,\bar{z}).\]
Since $\mathcal{F}\equiv F\mod K$, this equation implies
\[\mathcal{F}(z,\bar{z},\lambda)=\chi(g,\lambda)\mathcal{F}(z,\bar{z},\lambda), \hbox{ for all } z\in M,\ \lambda\in S^1.\]
Since  $\mathcal{F}$ is full, $\chi(g,\lambda)=e$ follows.

``(4) $\Rightarrow$ (2)": If $F_-$ is defined on $M$, then  $F_-(g.z,\lambda)=F_-(z,\lambda)$ for all $g\in \pi_1(M)$, $\lambda\in S^1.$ But then $F=F_-F_+$ satisfies
\[\begin{split}
F(g.z,\overline{g.z},\lambda)&=F_- (g.z,\lambda) F_+(g.z,\overline{g.z},\lambda)\\
&=F_- (z,\lambda) F_+(g.z,\overline{g.z},\lambda)\\
&=F(z,\bar{z},\lambda)F_+(z,\bar{z},\lambda)^{-1}F_+(g.z,\overline{g.z},\lambda).\\
\end{split}\]
The reality of $F(g.z,\overline{g.z},\lambda)$ and $F(z,\bar{z},\lambda)$ yields $F_+(z,\bar{z},\lambda)^{-1}F_+(g.z,\overline{g.z},\lambda)=k(g, z,\bar{z})$. Hence (2) follows.

``(3) $\Rightarrow$ (4)": Assume $C$ is a holomorphic extended frame defined on $M$. Then  $C(g.z,\lambda)=C(z,\lambda)B_+(g,z,\bar{z},\lambda)$
by assumption and $F_-=CS_+$ implies
\[F_-(g.z,\lambda)=C(z,\lambda)B_+(g,z,\bar{z},\lambda)S_+(g.z,\overline{g.z},\lambda)=F_-(z,\lambda)  T_+(g,z,\bar{z},\lambda).\]
But since $F_-=e+O(\lambda^{-1})$, this Birkhoff decomposition is unique and $T_+(g,z,\bar{z},\lambda)=e$ follows.
\end{proof}

\begin{corollary}
 { Let  $M $   be any Riemann surface and $G/K$ any inner semi simple  symmetric space and let $\mathcal{F}: M\rightarrow G/K$ be any totally symmetric harmonic map  in $G/K$.  Then the normalized extended frame  $F_-$,  is fixed under the action  of $\pi_1(M)$.
 In particular,  we have $F_- (g.z, \lambda) =   F_- (z, \lambda)$ for all $z \in \tilde{M}$ and $g \in \pi_1(M)$.
 As a consequence, the normalized potential  {$\eta_- = F_-(z,\lambda)^{-1} \dd F_-(z,\lambda)$,}  is fixed under the action  of each $g \in \pi_1(M)$.
}
\end{corollary}

If we assume $M$ to be non-compact then we obtain the following stronger result:
\begin{theorem}\label{thm-invariant-potential} We retain the notation and the assumptions stated throughout
this paper, but we assume in addition that $M$ is non-compact.

Then  for any harmonic map     $\mathcal{F}:M\rightarrow G/K$ into the
inner symmetric space $G/K$, compact or non-compact,  the following statements are equivalent:
 \begin{enumerate}
\item $\mathcal{F}_{\lambda}: \tilde{M} \rightarrow G/K$ has trivial monodromy for all $\lambda \in S^1$.
\item There exists some extended frame $F$ which satisfies
\[F(g.z,\overline{g.z},\lambda)=F(z,\bar{z},\lambda) k(g, z, \bar z)~~ \hbox{ for all }~~\lambda\in S^1,\ g\in \pi_1(M)
\hbox{ and some } k(g, z, \bar z) \in K.\]
\item There exists a holomorphic extended frame $C$ for the harmonic map $\mathcal{F}$ which satisfies
\[C(g.z,\lambda)=C(z,\lambda)\ \hbox{ for all }\  \lambda\in S^1,\ g\in\pi_1(M).\]
\item The integrated normalized potential $F_-$ of the harmonic map $ \mathcal{F}:M\rightarrow G/K$ satisfies
\[F_-(g.z,\lambda)=F_-(z,\lambda)\ \hbox{ for all }\  \lambda\in S^1,\ g\in\pi_1(M).\]
 \end{enumerate}
 \end{theorem}

\begin{proof}
From Proposition \ref{prop-frame} we know that $(1)$ and $(4)$ are equivalent.

 {
By the proof of part $(1)$ of Theorem \ref{thm-monodr} we know that in general there exists some holomorphic extended frame satisfying  $C(g.z,\lambda)=M(g,\lambda)C(z,\lambda)$ for all $g \in \pi_1(M).$
Therefore,  by using such a $C(z,\lambda)$ in $(3)$  of
Proposition \ref{prop-frame} we observe that the monodromy is trivial if and only if
$C(z,\lambda)$ is invariant under $\pi_1(M).$ Finally, $(2)$ is equivalent to  $(2)$ of
Proposition \ref{prop-frame}.}
\end{proof}

\begin{remark}
It is not known (except for some examples)  for which $M$ and which $G/K$ the factor $k$ in $(2)$ can be removed.
\end{remark}

  For   Willmore surfaces we consider the inner symmetric space
 $SO^+(1,n+3)/SO^+(1,3)\times SO(n)$. In this case the statements of
  Theorem \ref{thm-invariant-potential} hold, of course, for $M$ non-compact.

 \begin{theorem}
 If the inner symmetric space is $SO^+(1,n+3)/SO^+(1,3)\times SO(n),$ and if $M$ is compact and different from $S^2$, then the statements of Theorem \ref{thm-invariant-potential} hold if in item $(3)$ the notion ``holomorphic" is replaced by
 ``meromorphic".
 \end{theorem}
 \begin{proof}
 We only need to consider the case of $M$ compact and different from $S^2$. In this case
 Theorem \ref{thm-monodr} shows that there exists an invariant meromorphic potential for any
 harmonic map defined on $M$. Since we assume that the harmonic map is totally symmetric, the claim follows
 (by following the argument of Theorem \ref{thm-monodr} ).
 \end{proof}

The case of $M = S^2$ can be phrased like the case for all other compact Riemann surfaces, but the usual provisos
(concerning singular point of the extended frames) need to be admitted.

We close this subsection by an application of the Duality Theorem in \cite{DoWa13}. For the case $M = S^2$,  {we can perform the Theorem   to  $S^2 \setminus{ \{p_1\}}$ and $ S^2 \setminus{ \{p_2\}} $ respectively, in the same way as done in \cite[Section 3.2]{DoWa12}.}

\begin{theorem} \label{duality: totally symmetric}
We retain the general assumptions of this paper.
Let $G/K$ be an inner symmetric space of non-compact type and  $\tilde{U}/(\tilde{U}\cap \tilde{K}^{\C})$
its compact dual. Then we obtain for a ``dual pair"
$\mathcal{F} : M \rightarrow G/K$  and
$\tilde{\mathcal{F}}_U : \tilde{U} \rightarrow   \tilde{U}/(\tilde{U}\cap \tilde{K}^{\C})$
of harmonic maps as constructed in the Appendix:
\[\hbox{$\mathcal{F}$   is totally symmetric  $\Longleftrightarrow$  $\tilde{\mathcal{F}} $ is totally symmetric.
}\]
\end{theorem}

\begin{proof}
This follows quite directly from the crucial Iwasawa decomposition formula
$$F(z,\bar{z},\lambda)=F_{\tilde{U}}(z,\bar{z},\lambda) S_+(z,\bar{z},\lambda).$$
Assume $\mathcal{F}$ is totally symmetric, then we obtain from Definition \ref{def-uni}
 \[{F}(z, \bar z, \lambda) = {F}_{\tilde{U}}(g.z, g.{\bar z}, \lambda) k(g,z, \bar z)^{-1}  =
F_{\tilde{U}}(g.z,g.\bar{z},\lambda) S_+(g.z,g.\bar{z},\lambda) k(g, z, \bar z)^{-1}.\]
This implies
\[F_{\tilde{U}}(g.z,g.\bar{z},\lambda) = F_{\tilde{U}}(z,\bar{z},\lambda) L_+(z,\bar{z},\lambda)\]
for each $g \in \pi_1(M)$. From this it follows (by definition) that $\tilde{\mathcal{F}}_U $ is totally symmetric.
The converse follows analogously.
\end{proof}

\subsection{\bf{Dressing for Algebraic and Totally Symmetric Harmonic Maps}}

Dressing is a very useful operation which permits to construct new harmonic maps from a given one.
Before applying this to the harmonic maps in question we
recall the definition of ``dressing" (e.g. \cite{DPW}, see also \cite{Gu-Oh}, \cite{BP}, \cite{Do-Ha3}, \cite{Do-Ha5}, \cite{TU1}).

We retain the notation and the  general assumptions made for this paper.
The symmetric space $G/K$  is again assumed to be compact, or non-compact, but the Riemann
surface $M$ is assumed to be simply connected.

Let $\tilde{\mathcal{F}}:\tilde{M}\rightarrow G/K$  be a harmonic map, and  $F$  an extended frame satisfying
$F(z_0,\bar{z}_0,\lambda) = e$ for some base point $z_0$. Let $h_+\in\Lambda^+G^{\C}_{\sigma}$ and consider the Iwasawa splitting (on an open dense subset $\tilde{M}'$ of $\tilde{M}$ containing $z_0,$ on which all factors in the following equation do not have any singularities
)
\begin{equation}
\label{eq-dress1} h_+F=\hat{F}\hat{W}_+.
\end{equation}

Here we assume w.l.g. that $\hat{F}(z_0, \bar{z}_0, \lambda) = e$ holds.

Then $\hat{F}(z, \bar z,\lambda)$ defines a  family of harmonic maps on $\tilde M':$
  \begin{equation}
\label{eq-dress2}
 {\hat{\mathcal{F}}_\lambda : = h_+ \sharp {\mathcal{F}_\lambda}  : \tilde{M}' \rightarrow G/K,\
 \hspace{5mm}
\hat{\mathcal{F}}_{\lambda} ( z, \bar z) :=\hat{F} (z, \bar z,\lambda) \mod K . }
\end{equation}

\begin{lemma}
Using the notation introduced just above,  {$\hat{F}(z, \bar z,\lambda)$} is an extended frame of
the harmonic map $\hat{\mathcal{F}}: \tilde{M}' \rightarrow G/K,$
$\hat{\mathcal{F}}(z, \bar z,\lambda) = { (h_+\sharp\mathcal{F})(z, \bar z, \lambda)},$
and satisfies w.l.g.  {$\hat{F}(z_0, \bar z_0,\lambda) = e$}.
\end{lemma}

\begin{theorem} Let's retain the general assumptions of this paper for  $M$ and $G/K$.
\newline
$(1)$ If $M$ is simply connected and $\mathcal{F} : M \rightarrow G/K$ is an algebraic harmonic map,
then also each dressed harmonic map
$ (h_+\sharp\mathcal{F}): \tilde{M}' \rightarrow G/K,$ is algebraic.
\newline
$(2)$ If  $\mathcal{F} : M \rightarrow G/K$ is a totally symmetric  harmonic map,
then also each dressed harmonic map
$ (h_+\sharp\mathcal{F}): \tilde{M}' \rightarrow G/K,$ is totally symmetric.
\end{theorem}
\begin{proof}
We first consider the case $M \neq S^2$.
The first claim follows, since equation (\ref{eq-dress1}) implies that $F$ contains only finitely many negative
powers of $\lambda$ if and only if $\hat{F}$ does.
To verify the second claim we consider some point $z' \in M'$ and denote by $V$ an open  neighbourhood
of $z'$ in $M'$. Let $p \in V$ and $g \in \pi_1(M)$. Then we all factors in the dressing equation
$ h_+F=\hat{F}\hat{W}_+$  are smooth near $p$. Moreover, applying $g$ we obtain
$(h_+ \chi h_*^{-1}) \hat{F} W_+ =  (h_+ \chi h_*^{-1}) h_+ F = g^* \hat{F} g^* \hat(W)_+$ on an open dense subset
about $p$. Since the given harmonic map is totally trivial, the monodromy matrix $\chi(\lambda) $ is trial, whence
$ \hat{F} W_+ =  g^* \hat{F} g^* {\hat(W)_+}$. This implies $\pi_1(M).M' = M'$ and that $ g^* \hat{F} =   \hat{F} k$ for some map $k$ into $K$. Hence the claim.

{For the case $M = S^2$,  we perform the above results on $M_1=S^2 \setminus{ \{p_1\}}$ and $ M_2=S^2 \setminus{ \{p_2\}} $ respectively, in the same way as done in \cite[Section 3.2]{DoWa12}. Note that to get $ (h_+\sharp\mathcal{F})$ well-defined on $S^2$, some additional equation need to be satisfied.}

\end{proof}








\section{\bf{Relating Extended Solutions to Extended frames}}

In \cite{BuGu}, the authors consider harmonic maps from connected compact Riemann surfaces
into connected, real, compact, semi simple Lie groups $G$ with trivial center. We will admit $G$ to be compact or non-compact. Similarly, our symmetric spaces also are admitted to be compact or non-compact.
 {We can assume w.l.g. that $G^\C$ is a semi simple simply-connected matrix Lie group and $G$ a subgroup of $G^\C$ \cite{Hoch}.}
 Harmonic maps into inner symmetric spaces are  then composed by some natural embedding so that the results about harmonic maps into Lie groups can be applied.

 We will recall this construction and describe the relation between extended solutions, used in \cite{BuGu}, and extended frames, used in the DPW approach.
While for the basic geometric expressions of DPW theory we have always used twisted loop groups (and will continue to do so), in \cite{BuGu} only untwisted loop groups have been used. So we will assume all along
 that ``DPW expressions" actually are contained in twisted loop groups, while generally expressions occurring in \cite{BuGu} will not be twisted.


\subsection{\bf{Relating harmonic maps into $\mathcal{F} :\rightarrow G/\hat{K}$ to modified harmonic maps
$\mathfrak{C}_h \circ \mathcal{F}: M \rightarrow G$}}

For the convenience of the reader and to fix notation we start with a simple remark.


Consider the inner compact symmetric space  $G/\hat{K}$  with inner involution $\sigma$, given by $\sigma(g) = h g h^{-1}$ and with $\hat{K} = Fix^\sigma (G)$. Note that $h^2 \in Center (G) = \{e\}. $  Then with $R_h(g)= gh$ we consider
the map
\begin{equation}\label{eq-Ch}
    \begin{tikzcd}[column sep=6mm,row sep=4mm]
G/\hat{K}  \ar{r}{\mathfrak{C}}    &   G     \ar{r}{R_h} & G \\
g   \ar{r}{}  &   g \sigma(g)^{-1} = ghg^{-1} h^{-1}  \ar{r}{} & \mathfrak{C}_h:=g h g^{-1}.
\end{tikzcd}
\end{equation}
 In this way $G/\hat K$ is an isometric, totally geodesic submanifold   of $G$ \cite{BuGu}, and ${\mathfrak{C}_h}$ will be called the ``modified Cartan embedding". Note that for outer symmetric spaces the above Cartan embedding does not apply directly.

Next we want to relate extended frame for a harmonic map  $\mathcal{F} : M \rightarrow G/\hat{K}$ to the extended frame of the modified  harmonic map $\mathfrak{C}_h\circ \mathcal{F}.$ So let's consider a harmonic map  $\mathcal{F}: M \rightarrow G/\hat{K}$. By $F : \tilde{M} \rightarrow  \Lambda G_{\sigma}$ we  denote the extended frame of $\mathcal{F}$ which is normalized to $F(z_0, \bar{z}_0, \lambda) = e$ at some fixed base point $z = z_0$ for all $\lambda \in S^1$.
The extended frame of a harmonic map $\mathcal{F}$ actually is for each fixed $\lambda$ the frame of the corresponding immersion $\mathcal{F}_\lambda$ of the
associated family of $\mathcal{F}$. Obviously, the twisting condition in our case means
\begin{equation} \label{wild}
{\sigma(\gamma)(\lambda) = h \gamma(-\lambda)h^{-1} = \gamma(\lambda)}\hbox{
for all $\gamma \in \Lambda G_{\sigma}$.}
\end{equation}
Next we consider the composition of the family of harmonic maps  $\mathcal{F}_\lambda$ with the
modified Cartan embedding $\mathfrak{C}_h$.  In our setting, since
$\mathcal{F}_\lambda =F( z, \bar{z}, \lambda)\mod \hat{K}$,   this yields the
$\lambda-$dependent harmonic map $\mathfrak{F}^h_\lambda$ given by
 \begin{equation} \label{harmonicrelation}
  \mathfrak{F}^h_{\lambda}=
 F( z, \bar{z}, \lambda) h F( z, \bar{z}, \lambda))^{-1}.
 \end{equation}
Note that $\mathfrak{F}^h_\lambda $  is a $\lambda-$dependent harmonic map satisfying
$(\mathfrak{F}^h_\lambda )^2=e, $ where the square denotes the product in the group $G$.
Moreover, we also  have     $\mathfrak{F}^h_\lambda (z_0, \bar{z}_0,\lambda)  = h.$ Harmonic maps into $G$ satisfying these two properties will be called ``modified harmonic maps".

Below we will use the notation:
$$\mathbb{F}(z, \bar z, \lambda) = \mathfrak{F}^h_{\lambda} (z, \bar z, \lambda) =
    F( z, \bar{z}, \lambda) h F( z, \bar{z}, \lambda))^{-1}.$$

\begin{theorem}
We retain the notation and the assumptions made just above. In particular, $z_0$ is a fixed basepoint in the Riemann surface $M$ Then
there is a bijection between harmonic maps  $\mathcal{F} : M \rightarrow G/\hat{K}$  satisfying
$\mathcal{F}(z_0,\bar{z}_0,\lambda) = e\hat{K}$ and modified harmonic maps, i.e. harmonic maps
  $\mathfrak{F}^h_{\lambda}: M \rightarrow G_h = \{ ghg^{-1}; g \in G \}$ satisfying
   $(\mathfrak{F}^h_{\lambda})^2 = e$ and  $\mathfrak{F}^h_{\lambda}(z_0, \bar{z}_0, \lambda) = h$ for all   $\lambda \in \C^*$.
  This relation is given by composition with the modified  Cartan embedding (and its inverse respectively).
  \end{theorem}

\begin{proof}
We have shown ``$\Longrightarrow$" above.
  Assume now we have  a harmonic map
$\mathbb{F}: M \rightarrow  {{G}_h}$ satisfying
 $\mathbb{F}^2 = e$ and  $\mathbb{F}(z_0, \bar{z}_0, \lambda) = h$ for all  $\lambda \in \C^*$.
   Since   $\mathfrak{C}_h$ is an isometric  diffeomorphism   onto its image, we consider
    (for all $\lambda \in \C^*$) the $\lambda$-dependent harmonic map
    $\mathcal{F}_\lambda = (\mathfrak{C}_h)^{-1} \circ \mathbb{F}_\lambda : M \rightarrow G/\hat{K}$.
    Clearly, then we have  $\mathcal{F}(z_0, \bar z_0, \lambda) = e\hat{K}$.  Moreover, by what was shown above, we now infer
    $\mathbb{F}(z, \bar z, \lambda) = \mathfrak{F}^h_{\lambda} (z, \bar z, \lambda) =
    F( z, \bar{z}, \lambda) h F( z, \bar{z}, \lambda))^{-1}$, where $F$ denotes the extended frame of $\mathcal{F}$.
    Since     $\mathcal{F} (z_0, \bar{z}_0, \lambda) = e$,
    $h =   F( z_0 \bar{z}_0, \lambda) h F( z_0, \bar{z}_0, \lambda))^{-1}= \mathfrak{F}^h_{\lambda} ( z_0, \bar{z}_0, \lambda)$ holds.
    \end{proof}


\subsection{\bf{Extended solutions for harmonic maps into Lie groups}}

 In this subsection, we  compare/unify the notation used in \cite{Uh}, \cite{BuGu} and \cite{DPW}.
For a harmonic map  $\mathbb{F}:\D\rightarrow G$, in \cite{Uh},  \cite{BuGu} ``extended solutions" are considered, while in \cite{DPW} always ``extended frames'' are used.  In this subsection we will explain the relation between these methods. Here we include primarily  the details which we will need to use.

We  have shown, among other things, in the  subsections above  that it is essentially sufficient
for our purposes to  consider  harmonic maps into Lie groups $G$.
In this subsection we consider harmonic maps into Lie groups following the approach of
\cite{Uh} and \cite{BuGu}.
 We start by relating the different loop parameters
used in \cite{Uh} , \cite{BuGu} and \cite{DPW} respectively to each other.
Note that in our work results derived from \cite{BuGu} assume to begin
with that $G$ and $G/K$ respectively is compact, but that the result used extends to the case of non-compact $G$ and $G/K$ respectively by the duality theorem. Also note that $M$ is assumed to be compact or non-compact.

 To begin with, we  recall the definition of {\em extended solutions}
following Uhlenbeck \cite{Uh},\cite{BuGu}.
Let  $\mathbb{F}: M  \rightarrow G$  a harmonic map and  let  $\tilde{M}$ be the
universal cover of $M$.
Set
\[\mathbb{A}=\frac{1}{2} \mathbb{F}^{-1}\mathrm{d} \mathbb{F} =\mathbb{A}^{(1,0)}+\mathbb{A}^{(0,1)},\]
where the superscript denotes ``$(1,0)-$part" and  ``$(0,1)-$part" respectively.
Consider {for $\tilde{\lambda} \in \C^*$} the equations
\begin{equation}\label{eq-Uh1}
\left\{
\begin{split}
\partial_z \Phi \mathrm{d}z&=(1-\tilde{\lambda}^{-1})\Phi\mathbb{A}^{(1,0)},\\
\partial_{\bar{z}} \Phi \mathrm{d} \bar{z}&=(1-\tilde{\lambda})\Phi\mathbb{A}^{(0,1)}.\\
\end{split}
\right.
\end{equation}
with $\Phi:  \tilde{M} \rightarrow \Omega G$, where the corresponding loop parameter
 is denoted here by $\tilde{\lambda},$ and $\Omega G$ denotes all elements $g(\tilde{\lambda}) \in \Lambda G$ which satisfy
 $g(\tilde{\lambda} = 1) =e.$
 Then, by Theorem 2.2 of \cite{Uh} (Theorem 1.1 of \cite{BuGu}),
there exists a solution $\Phi(z,\bar{z},\tilde{\lambda})$ to the above
equations such that
\begin{equation} \label{cond for Phi}
\Phi(z,\bar{z},\tilde{\lambda}=1)=e,~\ \mbox{and}
~ \Phi(z,\bar{z},\tilde{\lambda}=-1)= \mathbb{F} (z, \bar z)
\end{equation}
hold. This solution is unique up to multiplication by some
$\gamma \in \Omega G = \{ g \in \Lambda G^{\C}_\sigma |,  g(\lambda = 1) = e \}$ satisfying $\gamma (-1) = e$. Such solutions $\Phi$ are said to be {\bf extended solutions}.


\begin{lemma}\label{lemma-es} \cite{BuGu}  Let $\Phi(z,\bar{z}, \tilde{\lambda})$ be
 {\em an
extended solution} of the harmonic map  $\mathbb{F}:  \tilde{M} \rightarrow G$. Let $\gamma\in \Omega G$.
Then $\gamma(\lambda)\Phi(z,\bar{z},\tilde{\lambda})$ is an extended solution of the
harmonic map $\gamma(-1)\mathbb{F}(z,\bar{z})$.
\end{lemma}

\begin{remark}
Note that for extended solutions we do not need a special proviso, as we need for extended frames.
\end{remark}

\subsubsection{\bf{Describing harmonic maps into $G$, when considering $G$ as  symmetric space}}

	Next we show how the ``DPW approach" without the basepoint assumption  \cite{DPW} naturally leads to
	Uhlenbeck's extended solutions \cite{Uh}.
	We follow Section 9 of \cite{Do-Es} and consider the Lie group $G$  as the outer symmetric space
	$G = (G \times G)/ \Delta,$
where the defining symmetry $\tilde{\sigma}$ is given by
	$$\tilde{\sigma} (a,b) = ( b,a)$$ and we have  $$\Delta  = \{ (a,b) \in  G \times G; a=b \}.$$
	
	For the purposes of the DPW approach it is necessary to consider the $G \times G-$loop group
	$\Lambda(G \times G)$
	twisted by $\tilde{\sigma}$. We thus consider the automorphism of the loop group
	$\Lambda(G \times G) = \Lambda G \times \Lambda G$ given by
	$$\hat{\tilde{\sigma}} ((a,b)) (\lambda) = \tilde{\sigma}( a(\lambda), b(\lambda)) =
	(b(\lambda), a(\lambda)).$$

It is straightforward to verify that the twisted loop group $\Lambda(G \times G)_{\tilde{\sigma}}$
	is given by
	$$\Lambda(G \times G)_{\tilde{\sigma}} = \{ (g(-\lambda), g(\lambda)) ; g(\lambda) \in \Lambda G \}
	 \cong \Lambda G. $$

	Let's consider now a harmonic map $\mathbb{F}: M \rightarrow G.$
	Then the map
	$\mathring{\mathfrak{F}}:M \rightarrow G \times G$, given by
	 $\mathring{\mathfrak{F}}(z, \bar{z}) = \left(\mathbb F(z,\bar{z}),e\right)$,  is a global frame of $\mathbb{F}$.
	
	Following \cite{DPW} one needs to decompose the Maurer-Cartan form
	$\mathfrak{A} = \mathring{\mathfrak{F}}^{-1} \mathrm{d} \mathring{\mathfrak{F}}$ of $\mathring{\mathfrak{F}}$  into the eigenspaces of
	$\tilde\sigma$ and to introduce the loop parameter $\lambda$.
	One obtains (see \cite{Do-Es}, formula $(68)$):
	\begin{equation}
	\mathfrak{A}_{\lambda} = \left( \left(1+\lambda^{-1}\right) \mathbb A^{(1,0)} +  \left(1+\lambda\right)\mathbb A^{(0,1)} ,\ \left(1-\lambda^{-1}\right) \mathbb A^{(1,0)}  +  \left(1-\lambda\right) \mathbb A^{(0,1)} \right).
	\end{equation}

	\begin{theorem}
		Let  $G$  be a connected, compact or non-compact, semi-simple real Lie group with trivial center.
		Let  $\mathbb{F}: M \rightarrow G$ be a harmonic map.  Then, when  representing $G$ as the
		symmetric space  $G = (G \times G)/\Delta,$
		any extended  frame $\mathring{\mathfrak{F}} : \tilde{M}  \rightarrow \Lambda(G \times G)_{\tilde\sigma}$
		of $\mathbb{F}$ satisfying
		$\mathring{\mathfrak{F}}(z, \bar{z}, \lambda = 1) = \left(\mathbb F(z,\bar{z}),e\right)$
		is given  by a pair of functions,
		\[\mathring{\mathfrak{F}}(z, \bar{z},\lambda) =  ( \Phi (z, \bar{z},-\lambda), \Phi (z, \bar{z},\lambda)),\]
		where the  matrix function $ \Phi (z, \bar{z},\lambda)$ is an  extended solution  for $\mathbb{F}$ in the sense of Uhlenbeck \cite{Uh} as introduced above.
	\end{theorem}
	
	\begin{proof}
		Since the two components of $\mathfrak{A} $ only differ by a minus sign in $\lambda$, any solution to
		the equation $\mathfrak{A}_\lambda  =
		\mathring{\mathfrak{F}}(z, \bar z, \lambda)^{-1} \mathrm{d} \mathring{\mathfrak{F}}(z, \bar z, \lambda)$
		is of the form $ \mathring{\mathfrak{F}}(z, \bar z, \lambda) =
		(B(-\lambda) \Psi(z, \bar z, - \lambda) , B(\lambda) \Psi(z, \bar z,  \lambda))$,
		where $\Psi$ solves the equations \eqref{eq-Uh1}. Moreover, we can assume w.l.g. that $\Psi$ satisfies the two conditions for extended solutions stated above for $\lambda = \pm1$.
		Now
		$\mathfrak{F}(z, \bar{z}, \lambda = 1) = \left(\mathbb F(z,\bar{z}),e\right)$ implies
		{$B(1)  = B(1) \Psi (z, \bar{z},\lambda= 1) =e$ and $B( -1)\Psi(z, \bar z, -1) =
		 B(-1) \mathbb{F}(z, \bar z) = \mathbb{F}(z, \bar z),$
		 whence $B(-1) = e$ follows.
		 Therefore, $\Phi (z, \bar z, \lambda)  = B(\lambda)  \Psi(z, \bar z, \lambda)$}
		yields the claim.
	\end{proof}

	Note that in this theorem no normalization is required. Moreover, the loop parameter used in \cite{Uh} is the same as the one used in \cite{DPW}. However, the matrix functions $B(\lambda)$ and  $\Phi (z, \bar z, \lambda)$
	are not uniquely determined which causes the DPW procedure to yield quite arbitrary potentials, not easily permitting any converse construction procedure.

\subsection{\bf{Extended solutions and extended frames for harmonic maps into inner symmetric spaces}} \label{414}

Consider as before a harmonic map  $\mathcal{F}: M \rightarrow G/\hat{K}$ into a symmetric space
with inner involution $\sigma$, given by $\sigma(g) = h g h^{-1}$ and with $\hat{K} = Fix^\sigma (G)$.
 As above we consider the modified harmonic map $\mathbb{F}: \tilde{M} \rightarrow G$ given by
 \begin{equation}
     \label{eq-FF}
     \mathbb{F}(z, \bar z, \lambda) =   \mathfrak{F}^h_{\lambda} (z, \bar z, \lambda) =
 F( z, \bar{z}, \lambda) h F( z, \bar{z}, \lambda)^{-1}.
 \end{equation}
  For this $\lambda-$family of harmonic maps $\mathbb{F}_\lambda$ we compute
 \begin{equation*}
 \begin{split}
\mathbb{A}&=\frac{1}{2}\mathbb{F}_{\lambda}^{-1}\mathrm{d}\mathbb{F}_{\lambda} \\
\ &=\frac{1}{2}\left(F(z,\bar{z},\lambda)hF(z,\bar{z},\lambda)^{-1}\right)^{-1}\mathrm{d}\left(F(z,\bar{z},\lambda)hF(z,\bar{z},\lambda)^{-1}\right)\\
\ &=\frac{1}{2}F(z,\bar{z},\lambda) h^{-1}\alpha_{\lambda} h F(z,\bar{z},\lambda)^{-1}-\frac{1}{2}\mathrm{d}F(z,\bar{z},\lambda)F(z,\bar{z},\lambda)^{-1}\\
\ &=\frac{1}{2}F(z,\bar{z},\lambda)\left(\alpha_{-\lambda}-\alpha_{\lambda}\right)F(z,\bar{z},\lambda)^{-1}\\
\ &=-F(z,\bar{z},\lambda)\left(\lambda^{-1}\alpha_{\mathfrak{p}}'+\lambda\alpha_{\mathfrak{p}}''\right)F(z,\bar{z},\lambda)^{-1}.\\
\end{split}
\end{equation*}
Following Uhlenbeck's approach we need to introduce a new ``loop parameter" $\tilde{\lambda}$ now and consider Uhlenbeck's differential equation \eqref{eq-Uh1} for
$\Phi(z,\bar{z},\lambda,\tilde{\lambda})$  on $\tilde{M}$ with conditions  (\ref{cond for Phi})   for $\tilde{\lambda} = \pm 1:$
\begin{equation}\label{eq-Uh2}
\left\{
\begin{split}
\partial_z \Phi \mathrm{d}z&=-\Phi(1-\tilde{\lambda}^{-1})\lambda^{-1}F(z,\bar{z},\lambda) \alpha_{\mathfrak{p}}' F(z,\bar{z},\lambda)^{-1}\\
\partial_{\bar{z}} \Phi\mathrm{d}\bar z&=-\Phi(1-\tilde{\lambda})\lambda F(z,\bar{z},\lambda) \alpha_{\mathfrak{p}}'' F(z,\bar{z},\lambda)^{-1}\\
\end{split}
\right.\end{equation}
From (\ref{eq-Uh2}) it is natural to consider the Maurer-Cartan form $\widetilde{\mathbb{A}} $ of $\Phi (z,\bar{z},\lambda,\tilde{\lambda})F(z,\bar{z},\lambda)$. One obtains:
\begin{equation}
\begin{split}
\widetilde{\mathbb{A}}&=(\Phi F)^{-1} \dd ( \Phi F)\\
&=F^{-1}\dd F+F^{-1}(\Phi^{-1}\dd  \Phi) F\\
&=\alpha_{\lambda}-F^{-1}\Phi^{-1}\left(\Phi(1-\tilde{\lambda}^{-1})\lambda^{-1}F  \alpha_{\mathfrak{p}}' F ^{-1}+\Phi(1-\tilde{\lambda})\lambda F \alpha_{\mathfrak{p}}'' F ^{-1}\right)F\\
&=\alpha_{\lambda}- \left( (1-\tilde{\lambda}^{-1})\lambda^{-1}  \alpha_{\mathfrak{p}}'  +(1-\tilde{\lambda})\lambda \alpha_{\mathfrak{p}}''\right)\\
&=\lambda^{-1}  \alpha_{\mathfrak{p}}'+\alpha_{\mathfrak k}+\lambda \alpha_{\mathfrak{p}}''- \left( (1-\tilde{\lambda}^{-1})\lambda^{-1}  \alpha_{\mathfrak{p}}'  +(1-\tilde{\lambda})\lambda \alpha_{\mathfrak{p}}''\right)\\
&=\tilde{\lambda}^{-1}\lambda^{-1}  \alpha_{\mathfrak{p}}'+\alpha_{\mathfrak k}+\tilde{\lambda}\lambda \alpha_{\mathfrak{p}}'' \\
&=\alpha_{\tilde{\lambda}\lambda}. \\
\end{split}
\end{equation}
From this we derive immediately the relation
\begin{equation}\label{eq-lawson}
F(z,\bar{z}, \lambda\tilde{\lambda})=A(\lambda, \tilde{\lambda}) \Phi (z,\bar{z},\lambda,\tilde{\lambda})F(z,\bar{z}, \lambda).
\end{equation}
Substituting here  $z = z_0$ we derive, in view of the normalization of $F$ at $z = z_0$:
\begin{equation}\label{eq-A}
A(\lambda, \tilde{\lambda}) = \Phi (z_0 ,{\bar{z}}_0,\lambda,\tilde{\lambda})^{-1}.
\end{equation}
In particular, setting $\lambda=1$ in \eqref{eq-lawson} we obtain
\begin{equation}
F(z,\bar{z}, \tilde{\lambda})=A(1, \tilde{\lambda}) \Phi (z,\bar{z},1,\tilde{\lambda})F(z,\bar{z},1).
\end{equation}

Setting $\tilde\lambda=-1$ in \eqref{eq-lawson} we obtain $A(\lambda, -1)=F(z,\bar{z},-\lambda)\left( \Phi (z,\bar{z},\lambda,-1)F(z,\bar{z},\lambda)\right)^{-1}$.
The twisting condition for $F$ is $ F(z, \bar z, \lambda)  =h F(z, \bar z, -\lambda) h^{-1}$, whence
\[\begin{split}
    A(\lambda, -1)&=F(z,\bar{z},-\lambda)\left(\Phi (z,\bar{z},\lambda,-1)F(z,\bar{z},\lambda)\right)^{-1} \\
&    = h^{-1}F(z, \bar z, \lambda) h F(z, \bar z, \lambda)^{-1}  \Phi (z,\bar{z},\lambda,-1)^{-1} \\
&=  \Phi (z_0,\bar{z}_0,\lambda,-1)^{-1},
\end{split}\] since the left side is independent of $z$ and $F(z_0, \bar z_0, \lambda) =e$.
But now (\ref{cond for Phi}) and \eqref{eq-FF}  yield that the last term above is equal to
$\mathbb{F}(z_0, \bar{z}_0, \lambda)^{-1} = h^{-1}.$ Hence we have
\begin{equation}
A(\lambda, -1)={h^{-1}.}
\end{equation}

In summary we obtain (by setting $\lambda = 1$ and replacing $\tilde{\lambda}$ by $\lambda$):
\begin{proposition}\label{cor-Phi-F}
The extended solution $\Phi$, and the $\sigma-$twisted  extended frame $F$ satisfy
\begin{equation}\label{eq-lawson2}
\Phi (z, \bar{z},1,\lambda) = A(1, \lambda)^{-1}F(z,\bar{z},  \lambda)F(z,\bar{z}, 1)^{-1}.
\end{equation}
In particular, $\Phi  (z, \bar{z},1,\lambda)$ is contained in the based loop group $\Omega G$. Moreover,
for $\lambda = -1$  we obtain the harmonic map $ \mathbb{F}(z,\bar z,1) =  \mathbb{F}(z,\bar z) $.
\end{proposition}
\begin{proof} {We need only to verify the last statement. In fact, we have
\begin{equation*}
\begin{split}
\Phi (z, \bar{z},1,-1)&={A(1,-1)^{-1} } F(z,\bar{z}, -1)F(z,\bar{z}, 1)^{-1}\\
&=h (h^{-1} F(z, \bar z,1) h) F(z,\bar{z}, 1)^{-1} \\
&=F(z,\bar{z}, 1)hF(z,\bar{z}, 1)^{-1}\\
&=\mathbb{F}(z,\bar z,1).
\end{split}
\end{equation*}}
\end{proof}

\begin{proposition} \label{algebraic-harmonic <--->ext. sol.}
Let $\mathcal{F}: M \rightarrow G/\hat{K}$ be a harmonic map with associated modified harmonic map
$\mathbb{F}: M \rightarrow G$ and extended frame $F$ and extended solution  $\Phi: \tilde{M} \rightarrow \Omega G$
respectively. Then the harmonic map $\mathcal{F}: M \rightarrow G/\hat{K}$ is algebraic if and only if
there exists some $B(\lambda) \in \Omega G$  such that  $ B \Phi$ is a Laurent polynomial in $\lambda$.
\end{proposition}

\begin{proof}
``$\Rightarrow$" If the extended frame $F$ is a Laurent polynomial in $\lambda$, then Proposition \ref{cor-Phi-F} implies that
$A(1, \lambda)\Phi(z,\bar{z},  \lambda)$ is a Laurent polynomial. Moreover, $A(1,1) = e.$

``$\Leftarrow$" Assume $B(\lambda) \Phi(z, \bar z, \lambda)$ is a Laurent polynomial with
$B\in \Omega G$. Then by Proposition \ref{cor-Phi-F}
 we infer that
$B(\lambda) \Phi(z, \bar z, \lambda) =
B(\lambda) A(1, \lambda)^{-1}F(z,\bar{z},  \lambda)F(z,\bar{z}, 1)^{-1}$
is a Laurent polynomial in $\lambda$.
Inserting $z= z_0$ we obtain that also $B(\lambda) A(1, \lambda)^{-1}$ is a Laurent polynomial. As a consequence,
also $F(z,\bar{z},  \lambda)F(z,\bar{z}, 1)^{-1}$ is a Laurent polynomial, whence the claim.
\end{proof}
Analogously we can discuss the property ``totally symmetric" which is equivalent to ``the monodromy representation
of the extended frame is trivial". For this we recall that in equations (\ref{eq-Uh1}) and (\ref{cond for Phi}) we have fixed an extended solution $\Phi$ for some harmonic map $\mathbb{F} : M \rightarrow G$ associated to a harmonic map $\mathcal{F}: M \rightarrow G/\hat{K}$.

\begin{proposition} \label{monodromyPhi}
{Let $\chi(g,\lambda)$ denote the monodromy representation of $\mathcal{F}$. Then we have}
\begin{equation}
    \label{eq-mono-Phi}
    g^*\Phi_\lambda (z, \bar z) =      (A(1, \lambda)^{-1} \chi(g, \lambda) A(1, \lambda))
\Phi_{\lambda} (z, \bar z) \chi(g,\lambda =1)^{-1}, \hspace{2mm} g \in \pi_1(M).
\end{equation}
\end{proposition}
\begin{proof}It is straightforward to derive \eqref{eq-mono-Phi} from the equations (\ref{eq-lawson2}).
\end{proof}

\begin{corollary}
Let $\mathcal{F}: M \rightarrow G/\hat{K}$ be a harmonic map with associated modified harmonic map
$\mathbb{F}: M \rightarrow G$ and extended frame $F : \tilde{M} \rightarrow \Lambda G_\sigma$ and
extended solution  $\Phi: \tilde{M} \rightarrow \Omega G$
respectively. Then the harmonic map $\mathcal{F}: M \rightarrow G/\hat{K}$ is totally symmetric if and only if {$\Phi$  is totally symmetric in the sence that the monodromy representation of $\Phi$ is trivial.}
\end{corollary}
\subsection{\bf{Harmonic maps of finite uniton number}}

In view of the definition on top of page 546 of \cite{BuGu} we put
\begin{definition}
Let $\mathcal{F}: M \rightarrow G/\hat{K}$ be a harmonic map with associated modified harmonic map
$\mathbb{F}: M \rightarrow G$ and extended frame $F$ and extended solution  $\Phi: \tilde{M} \rightarrow \Omega G$
respectively. Then the harmonic map $\mathcal{F}: M \rightarrow G/\hat{K}$ is said to be of
\emph{finite uniton number} if and only if
\begin{enumerate}[(U1)]
    \item The extended solution $\Phi$ has trivial monodromy representation.
    \item  There exists some $B(\lambda) \in \Omega G$  such that  $ B \Phi$ is a Laurent polynomial in $\lambda.$
\end{enumerate}
Moreover, we say that $\mathcal{F}$ has {\it finite uniton number $k$} if
(see \eqref{eq-alg-loop} for the definition of $ \Omega^k_{alg} G $)
 \begin{equation}\Phi(M)\subset \Omega^k_{alg} G ,\
\hbox{ and } \Phi(M)\nsubseteq \Omega^{k-1}_{alg} G .\end{equation}
 In this case we write  $r(\Phi)=k$ and the minimal uniton number of $\mathbb{F}$ is defined
as \[r(\mathbb{F}):=min\{r(\gamma  Ad(\Phi))| \gamma\in \Omega_{alg} Ad G \}.\]
\end{definition}

\begin{remark}
\begin{enumerate}
    \item   We refer to \cite{BuGu,Gu2002,Uh} for examples of harmonic maps of finite uniton number.
\item   In \cite{DoWa-fu2} a result of  \cite{BuGu} is translated and modified so as  to produce in the DPW formalism a
family of meromorphic potentials which yield exactly all harmonic maps of finite uniton number, when carrying out
the DPW formalism with initial condition $e$ at a fixed base point $z_0,$ as explained in Theorem \ref{thm-DPW}.
\end{enumerate}

\end{remark}

 {
The notion of finite uniton number harmonic maps is related to extended solutions. In this paper we usually use extended framings of harmonic maps and read off the notion of ``finite uniton type" from these extended frames. It is important to this paper that these two notions describe the same class of harmonic maps.}

\begin{proposition} \label{typeequivnumber}
$\mathcal F$ is a harmonic map of finite uniton type in $G/K$ if and only if
 {$\mathbb{F} =   \mathfrak{C}_h \circ \mathcal{F}$  given in  section 4.1.1, where $\mathfrak{C}_h$ is the modified  Cartan }embedding of $G/K$ into $G$, is a harmonic  map of finite uniton number.
\end{proposition}

\begin{proof}
$``\Rightarrow"$ By definition, $F(z,\bar z, \lambda\tilde\lambda)$ is a Laurent polynomial in $\lambda \tilde{\lambda}$ and hence also in $\tilde\lambda$. As a consequence of this and
\eqref{eq-lawson} we have that \[A(\lambda,\tilde\lambda)\Phi(z,\bar z, \lambda,\tilde\lambda)=F(z,\bar{z}, \lambda\tilde\lambda)F(z,\bar{z}, \lambda)^{-1}\]
is also a Laurent polynomial in $\tilde\lambda$.  Hence, by  the above definition
 $\mathbb{F}(z,\bar z,\lambda)$ is of finite uniton number.

 $``\Leftarrow"$  By definition, there exists some $\Phi(z,\bar z, \lambda,\tilde\lambda)$ being a Laurent polynomial in $\tilde\lambda$ and hence by \eqref{eq-A} $A(\lambda,\tilde\lambda)=\Phi (z_* ,{\bar{z}}_*,\lambda,\tilde{\lambda})^{-1}$ is also a Laurent polynomial in $\tilde\lambda$. As a consequence, by
\eqref{eq-lawson} we have that $F(z,\bar{z}, \lambda\tilde\lambda)$ is also a Laurent polynomial in $\tilde\lambda$ and hence also in $\lambda$, i.e., $\mathcal F$ is of finite uniton type.
 Consequently, $\mathbb{F}(z,\bar z,\lambda)$ is of finite uniton number.
  \end{proof}

\section{\bf{A result of Grame Segal and a related conjecture}}

\subsection{A result of Grame Segal }
In \cite{Segal}, Graeme Segal proved a result which was rephrased in \cite[Theorem 1.2]{BuGu},
and which we slightly adjust to our notation just below.
Note that the general assumptions there were like ours. But $M$ and $G$ were assumed to be compact.:

\begin{theorem}
Let $M$ be  compact Riemann surface and $G$ a compact connected real semi simple Lie group with trivial center.
Let $\mathcal{F}: M \rightarrow G/K$ be a harmonic map, where $G/K$ is a
compact, inner, semi simple symmetric space with $\mathbb{F}=\mathfrak{C}_h\circ \mathcal{F} : M \rightarrow G$ as associated modified harmonic map. Let $F$ denote an extended frame  of $\mathcal{F}$ and $\Phi$ an extended solution for $\mathbb{F}$ as in the previous section.  If $\Phi$ descends to an extended solution $\Phi: M \rightarrow \Omega G$, then there exists some
$B \in \Omega G$ such that $B\Phi$ is a Laurent polynomial.
\end{theorem}

In view of the duality theorem, our notation and the results of the last section, this can be rephrased as follows:

\begin{theorem}
Let $M$ be  compact Riemann surface and $G$ a (compact or non-compact)  connected real semi simple Lie group
with trivial center.
Let $\mathcal{F}: M \rightarrow G/K$ be a totally symmetric harmonic map, where $G/K$ is a
(compact or non-compact), inner, semi simple symmetric space.
Then $\mathcal{F}$ is an algebraic harmonic map.
\end{theorem}


\subsection{\bf{A related conjecture}}
The conjecture in question basically is the converse of Segal's result.
However, in addition, the Riemann surface is no longer assumed to be compact.
\vspace{5mm}

{{\bf{Conjecture}}
\emph{Let $M$ be a Riemann surface, compact or non-compact and $G$ a compact or non-compact connected real semi simple Lie group with trivial center.  Let $\mathcal{F}: M \rightarrow G/K$   be a an algebraic harmonic map, where $G/K$ is a  (compact or non-compact), real, inner, semi simple symmetric space.}
\emph{Then $\mathcal{F}$ is totally symmetric.}
\vspace{5mm}

We will finish this paper with a few remarks.
\begin{remark}
\begin{enumerate}
    \item
In view of the duality results Theorem \ref{duality:algebraic} and Theorem \ref{duality: totally symmetric}, we can assume w.l.g. that $G/\hat{K}$ is compact and thus can apply the results of \cite{BuGu}.
\item
Since $\mathcal{F}$ is defined on $M$, we infer that the lifted harmonic map
$\tilde{\mathcal{F}}: \tilde{M} \rightarrow G/\hat{K}$ is invariant under $\pi_1(M)$.
For the associated family $\tilde{\mathcal{F}}_\lambda$   of harmonic maps
this is equivalent to that its monodromy representation $\chi(g,\lambda) : \pi_1(M) \rightarrow \Lambda G_\sigma $
is trivial for $\lambda = 1$. In general,  harmonic maps defined on $M$ generally have non-trivial monodromy representation for almost all $\lambda \in S^1$ and $\C^*$ respectively.
\end{enumerate}
\end{remark}

\vspace{3mm}

{  \bf Acknowledgements}\ \
 {PW} was  supported by the Project  12371052 of NSFC.

{\footnotesize
\def\refname{References}

}
{\small\

 Josef F. Dorfmeister

Fakult\" at f\" ur Mathematik, TU-M\" unchen,

Boltzmannstr.3, D-85747, Garching, Germany

{\em E-mail address}: dorfmeis@gmail.de\\

Peng Wang

School of Mathematics and Statistics, FJKLMAA,

Key Laboratory of Analytical Mathematics and Applications (Ministry of Education),

Fujian Normal University, Qishan Campus,

Fuzhou 350117, P. R. China

{\em E-mail address}: {pengwang@fjnu.edu.cn}


\begin{thebibliography}{9}


 { \bibitem{Bungart} Bungart, L.,{\em  Analytic Fiber Bundles-I} Topology 7 (1966), 55-68.}

\bibitem{BuGu} Burstall,  F.E., Guest, M.A., {\em Harmonic two-spheres in compact symmetric spaces, revisited,} Math. Ann. 309 (1997), 541-572.

\bibitem{BP}
Burstall, F., Pedit, F., {\em Dressing orbits of harmonic maps,} Duke Math. J. 80 (1995), no. 2, 353-382.

\bibitem{BR} Burstall, F., Rawnsley, J.H. {\em Twistor theory for Riemannian symmetric spaces: with applications to harmonic maps of Riemann surfaces}. Lecture Notes in Mathematics, Vol. 1424, Springer, Berlin, 1990.



\bibitem{D:open cells} Dorfmeister, J.
{\em Open Iwasawa Cells and Applications to Surface Theory}
Variational Problems in Differential Geometry,
         London Math.Soc. LN 394, p. 56-67, Cambridge Univ. Press 2012.

\bibitem {Do-Es} Dorfmeister, J., Eschenburg, J.-H.
{\em Pluriharmonic Maps, Loop Groups and Twistor Theory}
Ann. Global Anal. Geom. Vol. 24, No.4, 301-321.



\bibitem {Do-Ha2}
Dorfmeister, J., Haak, G. {\em On symmetries of constant mean curvature surfaces. I. General theory,} Tohoku Math. J. (2)
50 (1998), 437-454.

\bibitem {Do-Ha3}
Dorfmeister, J., Haak G. {\em Investigation and application of the dressing action on surfaces of constant mean curvature, } Quart. J. Math., 2000, 50(1): 57-73.
\bibitem {Do-Ha5}
Dorfmeister, J., Haak, G. {\em  Construction of non-simply connected CMC surfaces via dressing}, J. Math. Soc. Japan 55, (2003), 335-364

\bibitem{DPW} Dorfmeister, J., Pedit, F., Wu, H., {\em Weierstrass type representation of harmonic maps
into symmetric spaces,} Comm. Anal. Geom. 6 (1998), 633-668.

\bibitem{DoWa11}  Dorfmeister, J.,  Wang, P. {\em Weierstrass-Kenmotsu representation of Willmore surfaces in spheres},   {Nagoya Math. J. 244 (2021), 35-59.}


\bibitem{DoWa12} Dorfmeister, J.,  Wang, P. {\em Willmore surfaces in spheres: the DPW approach via the conformal Gauss map.}  Abh. Math. Semin. Univ. Hambg. 89 (2019), no. 1, 77-103.

\bibitem{DoWa13}  Dorfmeister, J.,  Wang, P. {\em A duality theorem for harmonic  {maps into non-compact} symmetric spaces and compact symmetric spaces,  {submitted, https://arxiv.org/pdf/1903.00885.pdf}}

\bibitem{DoWa-fu2} Dorfmeister, J.,  Wang, P. {\em Harmonic maps of finite uniton type into inner symmetric spaces},
submitted.


\bibitem{Do-Wa-sym} Dorfmeister, J.,  Wang, P. {\em On symmetric Willmore surfaces in spheres I: the orientation preserving case,}  Differential Geom. Appl. 43 (2015), 102-129.


\bibitem{Gu2002}  Guest,  M.A. {\em An update on Harmonic maps of finite uniton number, via the Zero Curvature
Equation,} Integrable Systems, Topology, and Physics: A Conference on Integrable Systems
in Differential Geometry (Contemp. Math., Vol. 309, M. Guest et al., eds.), Amer. Math.
Soc., Providence, R. I. (2002), 85-113.



\bibitem{Gu-Oh}  Guest, M.A., Ohnita, Y.  {\em Group actions and deformations for harmonic maps }, Journal of the Mathematical Society of Japan, 1993, 45(4): 671-704.


 \bibitem {Esch-Ma-Qu} Eschenburg, J.-H., Mare, A.-L., Quast, P. {\em Pluriharmonic maps into outer symmetric spaces and a subdivision of Weyl chambers,} Bull. Lond. Math. Soc., 2010,  42(6): 1121-1133.
 \bibitem {Forster} Forster, O. {\em Lectures on Riemann Surfaces,} Springer-Verlag, Berlin, Heidelberg, New York, 1977.

\bibitem{Helein} H\'{e}lein, F.  {\em Willmore immersions and loop groups,}
 J. Differ. Geom., 50, 1998, 331-385.

 {\bibitem{Hoch} Hochschild,G. {\em The structure of Lie groups,} Holden-Day Inc., San Francisco, 1965.}

 \bibitem{Kac} Kac.V. {\em Infinite Dimensional Lie Algebras,} Cambridge University Press, Cambridge, 1985.


 \bibitem{Ke1} Kellersch, P. {\em  Eine Verallgemeinerung der Iwasawa Zerlegung in Loop Gruppen}, Dissertation, Technische Universit\"{a}t M\"{u}nchen, 1999.  http://www.mathem.pub.ro/dgds/mono/ke-p.zip.


\bibitem{PS} Pressley, A.N., Segal, G.B. {\em Loop Groups,} Oxford University Press, 1986.

\bibitem{Roehrl}  R\"{o}hrl, H. {\em Holomorphic fiber bundles over Riemann surfaces,} Bull. Am. Math. Soc. 68 (6) (1962) 125-160.

\bibitem{Segal} Segal, G. {\em Loop groups and harmonic maps},  LMS Lecture Notes Ser, 1989, 139: 153-164.

\bibitem{TU1} Terng, C-L., Uhlenbeck, K. {\em  Backlund transformations and loop group actions,} Comm. Pure. Appl. Math., 53(2000), 407-445.

\bibitem{Uh}Uhlenbeck, K. {\em Harmonic maps into Lie groups (classical solutions of the chiral model),} J.
Diff. Geom. 30 (1989), 1-50.


 \bibitem{Wang-1} Wang, P., {\em Willmore surfaces in spheres via loop groups II: a coarse classification of Willmore two-spheres by potentials}, arXiv:1412.6737.


\bibitem{Wang-3}Wang, P., {\em Willmore surfaces in spheres via loop groups IV: on totally isotropic Willmore two-spheres in $S^6$},  Chin. Ann. Math. Ser. B 42 (2021), no. 3, 383-408.



\bibitem{Wu} Wu, H.Y. {\em  A simple way for determining the normalized potentials for harmonic maps,} Ann.
Global Anal. Geom.  17 (1999), 189-199.


\end{thebibliography}
\end{document}